\documentclass[12pt]{article}
\usepackage{amssymb,latexsym,amsfonts,amsmath, amsthm}
\usepackage{enumerate}

\usepackage{graphicx}

\usepackage{tikz}

\usetikzlibrary{patterns}
\usetikzlibrary{calc}

\usepackage{enumitem}

\usepackage{ifthen}

\makeatletter
\@namedef{subjclassname@2010}{%
  \textup{2010} Mathematics Subject Classification}
\makeatother

\newtheorem{theorem}{Theorem}
\newtheorem{lemma}[theorem]{Lemma}
\newtheorem{cor}[theorem]{Corollary}
\newtheorem{prop}[theorem]{Proposition}

\theoremstyle{definition}

\newtheorem{claim}{Claim}

\pgfrealjobname{nsac}

\def\subqed{{\unskip\nobreak\hfil\penalty50
	\hskip1em\hbox{}\nobreak\hfil\raise2pt\hbox{\vrule height4pt width4pt depth0pt}
	\parfillskip=0pt \finalhyphendemerits=0 \par}}

\newcommand{\R}{\mathbb{R}}
\newcommand{\N}{\mathbb{N}}
\newcommand{\vx}{\mathbf{x}}
\newcommand{\vy}{\mathbf{y}}

\newcommand{\restrict}{\upharpoonright}

\newcommand{\concat}{\smallfrown}

\newcommand{\piset}[2]{\mathbf\Pi_{#1}^{#2}}
\newcommand{\sigmaset}[2]{\mathbf\Sigma_{#1}^{#2}}

\newcommand{\set}[2]{\left\{#1\;|\;#2\right\}}

\newcommand{\bd}[1]{\partial {#1}}

\newcommand{\cl}[1]{\overline{#1}}

\newcommand{\hideme}[1]{}

\begin{document}





\title{Strong Arcwise Connectedness}

\date{}

\begin{abstract}
A space is $n$-strong arc connected ($n$-sac) if for any $n$ points in the space there is an arc in the space visiting them in order. A space is $\omega$-strong arc connected ($\omega$-sac) if it is $n$-sac for all $n$.  We study these properties in finite graphs, regular continua, and rational continua. There are no $4$-sac graphs, but there are $3$-sac graphs and graphs which are $2$-sac but not $3$-sac. For every $n$ there is an $n$-sac regular continuum, but no regular continuum is $\omega$-sac. There is an $\omega$-sac rational continuum. For graphs we give a simple characterization of those graphs which are $3$-sac. It is shown, using ideas from descriptive set theory,  that there is no simple characterization of $n$-sac, or $\omega$-sac, rational continua.

Primary 54F15, 54H05; Secondary 54F50.

Analytic set, arc connected, cyclicly connected, ill founded trees, $n$-arc connected, $n$-strong arc connected, 
rational curve, regular curve, wadge reducible, $\mathbf{\Pi}_1^1$-complete, $\mathbf{\Pi}_2^1$-complete,
$\mathbf{\Sigma}_1^1$-complete, $\omega$-strong arc connected.
\end{abstract}

\maketitle

\section*{Introduction}

In \cite{ac} the property of being {\it $n$-arc connected} was introduced --- a topological space is $n$-arc connected if given any $n$ points in the space, there is an arc in the space containing the points. In this paper we strengthen the condition `there is an arc containing the points' by requiring the arc to traverse the points in a given order. We call this property {\it $n$-strong arc connectedness} (abbreviated $n$-sac), and we call a space which is $n$-sac for all $n$ an $\omega$-strongly arc connected space ($\omega$-sac). 

Evidently a space is $2$-sac if and only if it is arc connected. Many naturally occurring examples of arc connected spaces, especially those of dimension at least two, are $\omega$-sac. For instance it is easy to see that manifolds of dimension at least $3$, with or without boundary, and all $2$-manifolds without boundary, are $\omega$-sac. But note that the closed disk is $3$-sac but not $4$-sac (there is no arc connecting the four cardinal points in the order \emph{North}, \emph{South}, \emph{East} and then \emph{West}).
We are lead, then, to focus on one-dimensional spaces, and in particular on \emph{curves:} one-dimensional continua (compact, connected metric spaces). 

To further hone our focus, we observe that there is a natural obstruction to spaces being $3$-sac. Suppose a space $X$ contains an arc-cut point $x_1$ (in other words, $X \setminus \{x_1\}$ is not arc connected), and fix points $x_2$ and $x_3$ for which there is no arc in $X \setminus \{x_1\}$ from $x_2$ to $x_3$. Then 
no arc in $X$ visits the points $x_1, x_2, x_3$ in the given order, and thus $X$ is not $3$-sac. More generally, see Lemma~\ref{finite_cut_not_4sac}, if removing some $n-2$ points from a space renders it arc disconnected, then it is not $n$-sac. A continuum is said to be \emph{regular} if it has a base all of whose elements have a finite boundary, and is \emph{rational} if it has a base all of whose elements have a countable boundary. It is well known that all rational continua are curves. From our observation it would seem that regular curves could only `barely' be $n$-sac for $n \ge 3$, while rational curves could only `barely' be $\omega$-sac --- if, indeed, such spaces exist at all.

This paper investigates the $n$-sac and $\omega$-sac properties in graphs, regular curves, and rational curves.
The paper is divided into five sections, in the first section we formally introduce $n$-strong arc connectedness, give restrictions on spaces being $4$-sac, or more generally $n$-sac. In particular we show that no planar continuum is 
$4$-sac. In Section~2 we study $n$-strong arc connectedness in graphs noting that graphs are never $4$-sac, and giving a simple (in a precisely defined sense) characterization of those graphs which are $3$-sac. In Section~3 we observe that regular curves are never $\omega$-sac, but that there exist, for every $n$, a regular curve which is $n$-sac but not $(n+1)$-sac. While in Section~4 we construct a locally connected $\omega$-sac rational curve. In contrast to the case with graphs, we have not been able to find a simple characterization of regular $n$-sac curves. However, in the last section we study the complexity of the set of rational $n$-strongly arc connected continua as a subset of the space of all subcontinua of $\mathbb{R}^N$, for $N\geq 3$, and deduce that --- \emph{provably} --- there is no simple characterization of rational $n$-sac, or $\omega$-sac, curves. Further, we prove that there is no characterization of $n$-sac or $\omega$-sac curves (not necessarily rational) less complex than the definition itself. We introduce the machinery from descriptive set theory to make these claims precise, and to prove them, at the start of Section~5. The paper concludes with a discussion of open problems.

\section{Preliminaries}

In this section we introduce the basic definitions and notation used throughout the paper. Most of the basic notions are taken from \cite{nad}.

A topological space $X$ is {\it $n$-strongly arc connected} ($n$-sac) if for every distinct $x_1, \ldots , x_n$ in $X$ there is an arc 
$\alpha : [0,1] \to X$ and $t_1< t_2 < \cdots < t_n$ from $[0,1]$ such that $\alpha(t_i)=x_i$ for $i=1, \ldots , n$  --- in other words, 
the arc $\alpha$ `visits' the points in order. Note that we can assume that $t_1=0$ and $t_n=1$, or even that $t_i=(i-1)/(n-1)$ for $i=1, \ldots , n$.
A topological space is called $\omega$-sac if it is $n$-sac for every $n$. 

In connection with $n$-arc connectedness, observe that $n$-strong arc connectedness implies $n$-arc connectedness. On the other hand, a simple closed
curve is $\omega$-arc connected but is not $4$-strongly arc connected, thus the class of $n$-strongly arc connected spaces is a proper subclass of $n$-arc
connected spaces.

\begin{lemma}\label{finite_cut_not_4sac}  Let $X$ be a topological space.
If there is a finite $F$ such that $X \setminus F$ is disconnected, then $X$ is not $(|F|+2)$-sac.
\end{lemma}
\begin{proof}
If $F$ is empty then $X$ is disconnected and hence not $2$-sac. So suppose $F=\{x_1, \ldots, x_n\}$ for $n\ge 1$. Let $U$ and $V$ be an open partition of $X \setminus F$. Pick $x_{n+1}$ in $U$ and $x_{n+2}$ in $V$. Consider an arc $\alpha$  in $X$ visiting $x_1, \ldots , x_n,$ and then  $x_{n+1}$. Then $\alpha$ ends in $U$ and can not enter $V$ without passing through $F$. Thus no arc extending $\alpha$ can end at $x_{n+2}$ --- and $X$ is not $n+2$-sac, as claimed.
\end{proof}

\begin{cor}\label{finite_cut_cor} Let $X$ be a topological space.

(1) If there is an open non-dense set $U$ with  finite boundary, then $X$ is not $(|\bd{U}|+2)$-sac.

(2) A continuum containing a free arc is not $4$-sac.

(3) No compact continuous injective image of an interval is $4$-sac. 
\end{cor}
\begin{proof}
(1) is simply a restatement of Lemma~\ref{finite_cut_not_4sac}. 
For (2), apply (1) to an open interval inside the free arc. 
While for (3) note that, by Baire Category, a compact continuous injective image of an interval contains a free arc, so apply (2).
\end{proof}

Call an arc $\alpha$ in a space $X$ a `{\it no exit arc}' if every arc $\beta$ containing the endpoints of $\alpha$, and meeting $\alpha$'s interior must contain all of $\alpha$.
\begin{lemma}\label{no_exit_arc_not_4sac}
If a space contains a no exit arc then it is not $4$-sac.
\end{lemma}
\begin{proof}
Let $x_1$ and $x_2$ be the endpoints of $\alpha$. Pick $x_3$ and $x_4$ so that $x_1, x_3, x_4, x_2$  are in order along $\alpha$. Suppose, for a contradiction, $\beta$ is an arc visiting the $x_i$ in order. Since $x_3$ and $x_4$ are in the interior of $\alpha$, by hypothesis, $\beta$ contains $\alpha$. Now we see that if $\beta$ enters the interior of $\alpha$ from $x_1$ then it visits $x_3$ before $x_2$. While if $\beta$ enters the interior of $\alpha$ from $x_2$ it visits $x_4$ before $x_3$. Either case leads to a contradiction.
\end{proof}

\begin{prop}\label{planar_ctm_not_4sac}
No planar continuum is $4$-sac.
\end{prop}
\begin{proof}
Let $K$ be a plane continuum. If it is not arc connected then it is not $2$-sac, so suppose $K$ is arc connected. Pick $\vx_-$ (respectively, $\vx_+$) in $K$ to have minimal $x$--coordinate (resp., maximal $x$--coordinate). If $\vx_-$ and $\vx_+$ have the same $x$--coordinate, then $X$ is an arc, and so not $3$-sac.

Otherwise, translating the mid point between $\vx_-$ and $\vx_+$ to the origin, shearing in the $y$--coordinate only to move $\vx_-$ and $\vx_+$ onto the $x$--axis, and then scaling, we can assume without loss of generality that $\vx_-=(-1,0)$, $\vx_+=(+1,0)$ and $K \subseteq [-1,1] \times \R$.

There is an arc $\alpha$ in $K$ from $\vx_-$ to $\vx_+$. Some subarc, $\alpha'$, of $\alpha$ meets $\{-1\}\times \R$ and $\{+1\} \times \R$ in just one point (each). If for every $x$ in the interval $(-1,1)$ the vertical line $\{x\} \times \R$ meets $K$ in just one point, then $\alpha'$ is a free arc, and $K$ is not $4$-sac, as claimed.

Otherwise there is an $x_0 \in (-1,1)$ such that there are two distinct points $\vx_3$ and $\vx_4$ in $K \cap (\{x_0\} \times \R)$. We can suppose $\vx_3$ has minimal $y$--coordinate, $y_3$, while $\vx_4$ has maximal $y$--coordinate, $y_4$. Assume, for a contradiction, that there is an arc $\beta$  from $\vx_1=\vx_-$ to $\vx_4$ visiting $\vx_2=\vx_+$ and $\vx_3$ in order. Let $\beta_1$ be the subarc of $\beta$ from $\vx_1$ to $\vx_2$ and $\beta_3$ be the subarc of $\beta$ 
from $\vx_3$ to $\vx_4$. Note that $\beta_1 \cap \beta_3 = \emptyset$, and so $\beta_1$ meets $\{x_0\} \times \R$ only inside $\{x_0\} \times (y_3, y_4)$. Hence the line $L= (-\infty,-1)\times \{0\} \cup \beta_1 \cup (+1, +\infty) \times \{0\}$ splits the plane into two disjoint open sets, $U_3$ containing $\vx_3$, and $U_4$ containing $\vx_4$. However $\beta_3$ is supposed, on the one hand, to be an arc from $\vx_3$ to $\vx_4$, and so must cross $L$, and on the other hand, is forced to be disjoint from each part of $L$: $\beta_1$ (by choice of $\beta$) and both $(-\infty,-1)\times \{0\}$ and $(+1, +\infty) \times \{0\}$ (since $K \subseteq [-1,1] \times \R$) --- contradiction.
\end{proof}

\section{Graphs}

 From Corollary~\ref{finite_cut_cor}~(2) it is immediate that no graph is $4$-sac. Since only connected graphs will be considered, all graphs are $2$-sac. In this section we give a characterization of $3$-sac graphs. In fact, 
we show that for a general continuum $X$ the property of being $3$-sac is equivalent to the intensively studied property of being cyclicly connected (any two points in $X$ lie on a circle).

We begin this section by noticing that the triod and the figure eight continuum are not $3$-sac, while 
the circle and the theta curve continuum are $3$-sac. 
In \cite[Theorem 1]{bl}, Bellamy  and Lum proved:

\begin{theorem}
For a continuum $X$, the following are equivalent:
\begin{enumerate}
\item[(1)] $X$ is cyclicly connected;

\item[(2)] $X$ is arc connected, has no arc-cut point, and has no arc end points.
\end{enumerate}
\label{bell}
\end{theorem}



Using the previous theorem we obtain the following characterization of $3$-sac continua.

\begin{prop}\label{char_3sac}
For a continuum $X$, the following are equivalent:
\begin{enumerate}
\item[(1)] $X$ is cyclicly connected;

\item[(2)] $X$ is 3--sac;

\item[(3)] Any three points in $X$ lie either on a circle or on a theta curve.
\end{enumerate}
\end{prop}

\begin{proof} $(3) \Rightarrow (2)$: this follows from the fact that the circle and theta curve are both $3$-sac. 

$(2) \Rightarrow (1)$: if $X$ is $3$-sac, then for any $x\in X$, there is an arc that contains $x$ in its interior. So $X$ has no endpoints. If $X$ has an arc-cut point then, by Lemma 1,  $X$ is not $3$-sac. Now by Theorem~\ref{bell} $X$ is cyclicly connected. 

$(1) \Rightarrow (3)$: Let $x,y,z \in X$. By (1) there is a circle $C$ in $X$ containing $x$ and $y$. If $z\in C$ we are done. If $z\notin C$, then by (1) there are two arcs, $\alpha$ and $\beta$, from $z$ to $x$, that only meet at the endpoints. Let $a$ and $b$ be the points when $\alpha$ and $\beta$ first intersect $C$ and let $\alpha'$ and $\beta'$ be parts of $\alpha$ and $\beta$ from $z$ to $a$ and $b$ respectively. If $a\neq b$, $C\cup \alpha' \cup \beta'$ is the desired theta curve. If $a=b$, then $a$ is not  an arc-cut point by the above theorem, so there is an arc $\gamma$ from $z$ to some point on $C$ other than $a$ that misses $a$. Let $\gamma'$ be part of $\gamma$ that starts in $(\alpha' \cup \beta')-\{ a\}$, ends in $C-\{ a\}$ and does not meet $C\cup \alpha' \cup \beta'$ otherwise. Then $C\cup \alpha' \cup \beta' \cup \gamma'$ contains a theta curve that passes through $x,y,z$. 
\end{proof}

Notice that for finite graphs cyclicly connected is equivalent to having no cut points.

\begin{cor}
For a finite graph $X$, the following are equivalent:

\begin{enumerate}

\item[(1)] $X$ has no cut points;

\item[(2)] $X$ is $3$-sac;

\item[(3)] Any three points in $X$ lie either on a circle or on a theta curve.
\end{enumerate}
\end{cor}

It follows immediately from this characterization of $3$-sac graphs, combined with Lemma~3.2 and the proof of Proposition~3.4 of \cite{ac} that the set of $3$-sac graphs, considered as a subspace of $\mathcal{C}(I^3)$, the hyperspace of all subcontinua of the cube, is the intersection of a $G_{\delta \sigma}$ set (countable union of  countable intersections of open sets) and a $F_{\sigma \delta}$ set (countable intersection of countable unions of closed sets), but is neither a $G_{\delta \sigma}$ set nor a $F_{\sigma \delta}$ set. 

An alternative way of expressing this fact, is that there is a characterization of $3$-sac graphs of the logical form $\forall p \exists q \forall r \, \theta(p,q,r) \land \exists p \forall q \exists q \, \phi(p,q,r)$, where the quantifiers run over \emph{countable} sets, and $\theta(p,q,r)$, $\phi(p,q,r)$ are simple (boolean) sentences. However no logically simpler description of the $3$-sac graphs exists. 
This should be contrasted with the fact that the \emph{definition} of $3$-sac --- `$\forall x_1, x_2, x_3 \in X \, \exists \text{ arc } \alpha \ldots$' --- the two quantifiers run over \emph{uncountable} sets. Thus the given characterization of $3$-sac graphs is significantly simpler than the definition.

\section{Regular Curves}

In this section we construct, for every $n\geq 3$, an $n$-sac regular continuum, then using the Finite Gluing Lemma (Lemma~\ref{finite_gluing}) we
show that for any $n\geq 2$ there is a regular continuum --- therefore rational --- that is $n$-sac but not $(n+1)$-sac. 

We start the section by introducing the basic elements needed to construct an $n$-sac regular continuum.

Fix $N \ge 3$. Suppose $v_1, \ldots , v_k$ are affinely independent points in $\R^{N-1}$. Denote by $\langle v_1, \ldots , v_k\rangle$ the convex span of $v_1$ through $v_k$. Then $\langle v_1, \ldots , v_k\rangle$ is a $k$-simplex. We call the points $v_1, \ldots, v_k$ the vertices of $\langle v_1, \ldots , v_k\rangle$. For any $i \ne j$, we call $\langle v_i , v_j \rangle$ the edge from $v_i$ to $v_j$, and we let $v_i \wedge v_j$ be the midpoint between $v_i$ and $v_j$. Note that the space of all edges, $\bigcup_{i < j \le k} \langle v_i, v_j \rangle$, of $\langle v_1, \ldots , v_k\rangle$, is a complete graph on the vertices $v_1, \ldots , v_k$.

Fix $v_1, \ldots, v_N$ affinely independent points in $\R^{N-1}$, for example let $v_1$ through $v_{N-1}$ be the standard unit coordinate vectors, and $v_N=\mathbf{0}$.
Define the operation $\mathop{Trix}$ taking a simplex $\langle v_1, \ldots , v_N\rangle$ and returning a set of simplices, $\{ \langle v_i, v_i \wedge v_j : i \ne j\rangle : i =1, \ldots , N\}$. Inductively define sets of $N$-simplices as follows: $\mathcal{T}^N_0 = \{ \langle v_1, \ldots , v_N\rangle\}$, and $\mathcal{T}^N_{m+1} = \bigcup_{S \in \mathcal{T}^N_m} \mathop{Trix}(S)$. Let $T^N_m = \bigcup \mathcal{T}^N_m$, and $T^N=\bigcap_m T^N_m$. Then $T^N$ is a regular continuum we call the $N$-trix. Observe that the $3$--trix is the Sierpinski triangle, and the $4$--trix is the tetrix (hence our name for these continua).

Some additional notation. Given a simplex $S=\langle v_1, \ldots , v_N\rangle$, let $\mathcal{T}_1=\mathop{Trix}(S)$, and $T_1 = \bigcup \mathcal{T}_1$. Take any element of $\mathcal{T}_1$, say $S_i=\langle v_i, v_i \wedge v_j : j \ne i\rangle$. Call the point $v_i$ the external vertex of $S_i$, and call the points $v_i \wedge v_j$, for $j \ne i$, the internal vertices of $S_i$. For any $S$ in $\mathcal{T}_1$, denote the external vertex of $S$ by $v(S)$. For any two elements, $S$ and $S'$ of $\mathcal{T}_1$, denote the (unique) internal vertex common to $S$ and $S'$ by $S \wedge S'$. Note that $\{S \wedge S'\} = S \cap S'$.
Further for any $x$ in $T_1$, fix an element, $S(x)$, of $\mathcal{T}_1$ containing $x$.

It is easy to verify directly, or by applying Theorem~\ref{bell} and Proposition~\ref{char_3sac} that all $N$-trixes are $3$-sac.
From Lemma~\ref{planar_ctm_not_4sac} we see that the $3$-trix (i.e. the Sierpinski triangle) is not $4$-sac. Rather unexpectedly, the $4$-trix (i.e. the tetrix) is also not $4$-sac. To see this consider the sequence of points $x_1=v_1 \wedge v_2 , x_2=v_3, x_3=v_2, x_4=v_1 \wedge v_3$.
Using computer calculations, we have verified that the $5$--trix is $5$-sac but not $6$-sac. It is not clear to the authors for which $n$ a given $N$-trix is $n$-sac, but we can show that there is no upper bound on the natural numbers, $n$, that can be realized.

\begin{lemma}\label{n_trix}
Fix $n \ge 3$. Let $N=6n^2+12n+1$. The $N$-trix is $n$-sac.
\end{lemma}
\begin{proof}
Let $T=T^N$, the $N$-trix, and --- since $N$ is fixed to be $6n^2+12n+1$ --- otherwise suppress the superscript $N$. Take any $n$ points in $T$, say $x_1, \ldots , x_n$. Then there is a minimum $m \ge 1$ such that the $x_i$ are in distinct simplices in $\mathcal{T}_m$. Further there is a maximum $m'$ so that all the $x_i$ are in the some simplex $S$ of $\mathcal{T}_{m'}$. If there is an arc in $S \cap T$ visiting the points in order, then that same arc visits the points in order inside $T$. So without loss of generality, we can suppose that $m'=0$,  $S=T_0$, and the points $x_1, \ldots , x_n$ (obviously) each lie in an element of $\mathcal{T}_1$, but not all in the same element. Now call $m$ the height of the points, $x_1, \ldots, x_n$. 

There are $N=6n^2+12n+1$ elements of $\mathcal{T}_1$. Each of the $n$ points, $x_i$, can only be in at most $2$ members of $\mathcal{T}_1$. Hence we can find a subset $\mathcal{E}$  of $\mathcal{T}_1$ such that $\mathcal{E}$ has at least $3n+1$ members, and no point $x_i$ is in any element of $\mathcal{E}$. The lemma now follows from the next claim, which we prove by induction on $m$.

\paragraph{Claim:} for each $m \ge 0$,  points $x_1, \ldots , x_n$ of height $m$, and subset $\mathcal{E}$ of $\mathcal{T}_1$, such that $|\mathcal{E}| > 3n$ and $\mathcal{E} \cap \{S(x_i) : i \le n\} = \emptyset$ (for any choice of $S(x_i)$), there is an arc $\alpha$ visiting the points $x_1, \ldots, x_n$ in order, and, for $i=1, \ldots, n$, disjoint arcs (called `{\it spurs}'), $\beta_i$ from $x_i$ to the external vertex, $v(E)$, of some $E$ in $\mathcal{E}$.

\smallskip

\noindent {\bf Base Step, $m=1$.} Since $m=1$, we can assume that the sets $S(x_i)$, for $i=1, \ldots , n$, are distinct.

Pick some $C$ in $\mathcal{E}$. Let $\mathcal{E}'=\mathcal{E} \setminus \{C\}$. For each $i \le n$, and $j=1,2,3$, pick distinct $E_{i,j}$ from $\mathcal{E}'$. For each $i \le n$, pick three disjoint arcs in $S(x_i)$: $\alpha^-_i$ from $x_i$ to $S(x_i) \wedge E_{i,1}$, $\alpha^+_i$ from $S(x_i) \wedge E_{i,2}$ to $x_i$, and $\beta_i$ from $x_i$ to $S(x_i) \wedge E_{i,3}$. Extend $\beta_i$ by following the edge in $E_{i,3}$ to $v(E_{i,3})$. These $\beta_i$ are the required `spurs'. Denote by
$\Omega$ the set of simplices, $E_{i,3}$, containing these spurs.

For $i < n$, let $\alpha_i$ be the arc formed by following the natural edges (of elements of $\mathcal{T}_1$) between these vertices in the prescribed order: $S(x_i) \wedge S(E_{i,1})$, $S(E_{i,1}) \wedge S(C)$, $S(C) \wedge S(E_{i+1,2})$ and $S(E_{i+1,2}) \wedge S(x_{i+1})$. Let $\alpha$ be the path obtained by following these arcs in the given order: $\alpha^-_1, \alpha_i, \alpha^+_2$, $\alpha^-_2, \alpha_2 , \alpha^+_3$, $\ldots , \alpha^-_i , \alpha_i , \alpha^+_{i+1}, \ldots$, and finally $\alpha^-_{n-1} , \alpha_{n-1}, \alpha^+_n$. Since all the vertices appearing in the definition of the $\alpha_i$s are distinct, $\alpha$ is a path which does not cross itself, and so is an arc, which, by construction, visits the points $x_1, \ldots, x_n$ in order.

\smallskip

\noindent {\bf Inductive Step.} We assume the claim is true when the points come from a level $<m$. Prove for points on level $m$. First observe that for any $S$ in $\mathcal{T}_1$, $S\cap \mathcal{T}_m$
is homeomorphic to $\mathcal{T}_{m-1}$.

For clarity we will use $I_l$ to denote the set $\{1, 2, \dots , l\}$.
Let $x_1, x_2, \dots , x_n$ be $n$ points in $T$ of height $m$ and $\{S^1, S^2, \dots , S^k \}=\{S(x_1), \dots, S(x_n) \}$. For each $i\in I_k$ let
$x_{(i,1)}, x_{(i,2)}, \dots , x_{(i,k_i)}$ be a reenumeration of all the $x_j$s in $S^i$ such that
if $x_t=x_{(i,s)}$ and $x_l=x_{(i,r)}$ then $(i,s) < (i,r)$ if and only if $t < l$. For all $S$ in $\mathcal{T}_1$, let $\mathcal{S}_{m-1}$ denote $S\cap \mathcal{T}_m$. In each $S^i$ pick $6n+1$-many simplices of 
$\mathcal{S}^i_1$ that do not contain any of $x_{(i,j)}$s, none of them share external 
vertices, and none contain the external vertex of $S^i$; this can be done since $\mathcal{S}^i_1$ consists of $6n^2+13n+1$ simplices
and $S^i$ contains at most $n-1$ elements of $\{x_1, x_2, \dots , x_n\}$. Let this set of simplices be $\mathcal{E}_i$.

In the next step we will choose the simplices that will allow us to construct an arc between two consecutive $x_is$, whenever they lie on different elements
of $\mathcal{T}_1$.

For each $i\in I_k$ let $\Upsilon_i$ be a set of simplices $Y_{(i,j)}$ in $\mathcal{T}_1$
given as follows:
\begin{enumerate}
\item For $j < k_i$,
\begin{enumerate}
\item If $x_{(i,j)}$ and $x_{(i, j+1)}$ are not consecutive points in $\{x_1, x_2, \dots , x_n\}$,
then pick $Y_{(i,j)}$ such that $$Y_{(i,j)}\not\in\{S^1, S^2,\dots , S^k\}\cup
\{\bigcup_{t=1}^{i-1}\Upsilon_t\}\cup\{Y_{(i,1)}, Y_{(i,2)}, \dots , Y_{(i, j-1)}\},$$
$Y_{(i,j)}\cap \{x_1, x_2, \dots , x_n\}=\emptyset,$ and such that $Y_{(i,j)}\wedge S^i$ lies
in an element of $\mathcal{S}_{m-1}$ different from the elements containing the $x_{(i,j)}$s and the elements of $\mathcal{E}_i$.

\item If $x_{(i,j)}$ and $x_{(i, j+1)}$ are consecutive points in $\{x_1, x_2, \dots , x_n\}$,
then do nothing.
\end{enumerate}

\item For $j=k_i$,
\begin{enumerate}
\item If $x_{(i,k_i)}\neq x_n$, then pick $Y_{(i,k_i)}$ as above, satisfying the conditions
on (a).

\item If $x_{(i,k_i)}= x_n$, then do nothing.

\end{enumerate}

\item For $j=1$ (in some cases we are selecting twice for $j=1$),

\begin{enumerate}
\item If $x_{(i,1)}\neq x_1$, then pick $Y_{(1,0)}$ as in (1), satisfying the conditions on (a).

\item If $x_{(i,1)}= x_1$, then do nothing.

\end{enumerate}

\end{enumerate}

Denote by $y_{(i,j)}$ the vertex $S^i\wedge Y_{(i,j)}$. 

Now, in each sequence 
$\{x_{(i,1)}, x_{(i,2)}, \dots , x_{(i,k_i)}\}$ insert the points $y_{(i,j)}$ (if they exist) as follows:
$y_{(i,0)}$ before $x_{(i,1)}$, and $y_{(i,j)}$ immediately after corresponding $x_{(i,j)}$. So, for each
$i\in I_k$, we have constructed a sequence $l_i$ in $S^i$ such that a point $y_{(i,j)}$ lies between
$x_{(i,j)}$ and $x_{(i,j+1)}$ only if $x_{(i,j)}$ and $x_{(i,j+1)}$ are not consecutive points on
$\{x_1, x_2, \dots , x_n\}$.
Observe that for each $i\in I_k$, the set of points $x_{(i,1)}, x_{(i,2)}, \dots x_{(i,k_i)}$ have, in 
$S^i$,  height at most $m-1$, hence, by the choice of $y$s, the sequence of points $l_i$ also has height 
at most $m-1$. 

By Inductive Hypothesis, applied to $S^i$, $l_i$ and $\mathcal{E}_i$, there is an arc $\alpha_i$ in $S^i$ 
visiting the points of $l_i$ in order, and disjoint spurs $\beta_a$ for each $a\in l_i$ to 
external vertices of some elements in $\mathcal{E}_i$.   

\smallskip

\noindent
\textbf{Construction of an arc} through $x_1, x_2, \dots , x_n$: 
Pick $C\in \mathcal{E}$ containing none of the $y$s or $x$s. 
For each $i\in I_{n-1}$, let $\gamma_i$ be the arc connecting $x_i$ to $x_{i+1}$ given as
follows: If  $x_i$ and $x_{i+1}$ are in the same $S^t$ then $\gamma_i$ is the subarc
of $\alpha_t$ connecting them. If not, then $x_i=x_{(p,j)}$, $x_{i+1}=x_{(r,l)}$ and 
$y_{(p,j)}, y_{(r,l-1)}$ exist. 
Let $\gamma_i=\gamma_i^1\cup\gamma_i^2\cup\gamma_i^3\cup\gamma_i^4\cup\gamma_i^5$, where
$\gamma_i^1, \gamma_i^2, \gamma_i^3, \gamma_i^4, \gamma_i^5$ are as follows:

\begin{enumerate}
\item $\gamma_i^1$ is the subarc of $\alpha_p$ from $x_{(p,j)}$ to $y_{(p,j)}$ if possible or 
else a spur from  $x_{(p,j)}$ to some vertex $u$ of $S^p$, whichever is unused yet. 
In any case, there is a simplex $U$ in $\mathcal{T}_1$ such that $u=S^p \wedge U$ or
$y_{(p,j)}=S^p\wedge U$. Let $\gamma_i^2$ be the edge in $U$ connecting $S^p\wedge U$ to
$U\wedge C$.

\item similarly as in (1), $\gamma_i^3$ is the subarc $\alpha_r$ from $x_{(r,l)}$ to $y_{(r,l-1)}$ 
if possible or else a spur from $x_{(r,l)}$ to some vertex v of $S^r$, whichever is unused yet. In
any case, there is a simplex $V$ in $\mathcal{T}_1$ such that $v=S^r \wedge V$ or
$y_{(r,l-1)}=S^r\wedge V$. Let $\gamma_i^4$ be the edge in $V$ connecting $S^r\wedge V$ to
$V\wedge C$.

\item let $\gamma_i^5$ be the edge in $C$ connecting $U\wedge C$ and $V\wedge C$. 
\end{enumerate}

Let $\alpha=\bigcup_{i=1}^{n-1}\gamma_i$. Because of how $y$s and spur destinations were picked, $\alpha$
is an arc that visits the points $x_1, x_2, \dots , x_n$ in order.

\smallskip

\noindent
\textbf{Construction of spurs} to external vertices of elements of $\mathcal{E}$: suppose we have 
constructed spurs for all $x_l$, $l<i$ and $x_i\in S^j$. If spur $\beta_{x_i}$ of $x_i$ in $S^j$ is not 
contained in $\alpha$ then it only intersects it at $x_i$, extend $\beta_{x_i}$ as follows:  let $v_i$ be 
the other endpoint of $\beta_{x_i}$. Let $V_i$ be the simplex in 
$\mathcal{T}_1\setminus (\{\bigcup_{i=1}^k\Upsilon_i\} \cup(\bigcup_{i=1}^{k}S^i) \cup C)$ 
that intersects $S^j$ at $v_i$. Pick any simplex $E_i$ in $\mathcal{E}\setminus C$ that has not been 
picked for 
previous spur constructions. The spur $\beta_i$ consist of $\beta_{x_i}$, 
followed by the edge in $V_i$ connecting $v_i$ and $E_i\wedge V_i$, and the edge in $E_i$ connecting 
$E_i\wedge V_i$ with $v(E_i)$. 

Suppose $\beta_{x_i}$ is contained in $\alpha$. Observe that in this case $x_i=x_{(j,r)}$ and $x_{(j,r-1)}$
are not consecutive points of $\{x_1, x_2, \dots , x_n\}$, otherwise $\beta_{x_i}$ would not be
contained in $\alpha$. Hence $y_{(j,r-1)}$ exists. Let $\delta$ be the subarc of $\alpha_j$ connecting
$x_{(j,r)}$ to $y_{(j,r-1)}$, let $\gamma$ be the subarc of $\alpha$ connecting $x_{(j,r-1)}$ to
$y_{(j,r-1)}$, and let $\omega$ be the other end point of the spur $\beta_{y_{(j,r-1)}}$. By construction,
$\alpha \cap \delta=\{x_{(j,r)}, y_{(j,r-1)}\}$ and $\alpha \cap \beta_{y_{(j,r-1)}}=\{y_{(j,r-1)}\}$.

Since the diameters of the simplices in $\mathcal{T}_t$ approach zero as $t$ increases,
there exists, for a sufficiently large $t$, a simplex $\Lambda$ in $\mathcal{S}^j_t$ with the 
following properties:
\begin{enumerate}
\item $y_{(j,r-1)}$ is a vertex of $\Lambda$,

\item $x_{(j,r)}, \ x_{(j,r-1)}, \ \omega\not\in \Lambda$,

\item $\Lambda \cap \alpha$ is connected, and

\item $\Lambda$ does not intersect any spur, except for $\beta_{y_{(j,r-1)}}$.
\end{enumerate}

By the choice of $\Lambda$, the arcs $\delta$, $\gamma$  and $\beta_{y_{(j,r-1)}}$ intersect $\Lambda$ 
at different vertices of $\Lambda$, say $a, \ b, \ c$ respectively. Then revise $\alpha$ to go form $b$ 
to $y_{(j,r-1)}$ through an edge of $\Lambda$ and let $\beta$ consist of the following parts: 
the subarc of $\delta$ from $x_i=x_{(j,r)}$ to $a$, followed by the edge in $\Lambda$ from $a$ to $c$,
and followed by the subarc of $\beta_{y_{(j,r-1)}}$ from $c$ to $\omega$. Now extend $\beta$ as in the
previous case to get the spur $\beta_i$ for $x_i$. 
\end{proof}

\begin{lemma}[Finite Gluing]\label{finite_gluing}
If $X$ and $Y$ are $(2n-1)$-sac, and $Z$ is obtained from $X$ and $Y$ by identifying pairwise $n-1$ different points of $X$ and $Y$, then $Z$ is $n$-sac 
but not $(n+1)$-sac.
\end{lemma}

\begin{proof}
Let $z_1,z_2,\dots , z_n$ be any $n$ points in $Z$. For each $i$, if $z_i \in X\setminus Y$ and $z_{i+1} \in Y\setminus X$ or $z_i \in Y\setminus X$ and $z_{i+1} \in X\setminus Y$, pick $z_{(i,i+1)}\in (X\cap Y )\setminus \{z_1,z_2,\dots , z_n, z_{(1,2)}, z_{(2,3)}, \dots , z_{(i-1,i)} \}$ (if $z_{(1,2)}, z_{(2,3)}, \dots , z_{(i-1,i)}$ were picked). This is possible since $|X\cap Y | = n-1$. Let $\mathcal{Z}$ be the sequence of $z_j$s with $z_{(i,i+1)}$s inserted between $z_i$ and $z_{i+1}$ whenever they exist. And let $\mathcal{Z_X}$ be the sequence derived from $\mathcal{Z}$ by deleting the terms that do not belong to $X$. Define $\mathcal{Z_Y}$ similarly. Since elements of $\mathcal{Z_X}$ come either from $\{ z_1,z_2,\dots , z_n \}$ or from $X\cap Y$, $|\mathcal{Z_X} | \le 2n-1$. Similarly, $|\mathcal{Z_Y} | \le 2n-1$. Let $\beta$ be an arc in $X$ going through elements of $\mathcal{Z_X}$ in order and $\gamma$ be an arc in $Y$ going through elements of $\mathcal{Z_Y}$ in order. Let $a_1, a_2, \dots , a_k$ be $z_{(1,2)}, z_{(2,3)}, \dots , z_{(n-1,n)}$ whenever they exist, respectively. Without loss of generality, suppose $z_1 \in X$. Define $\alpha$ to be the union of
the following arcs:

\begin{enumerate}
\item the subarc of $\beta$ from $z_1$ to $a_1$;
\item the subarc of $\gamma$ from $a_1$ to $a_2$;
\item the subarc of $\beta$ from $a_2$ to $a_3$ $\dots$
\item the subarc of $\beta$ or $\gamma$ (depending on whether $k$ is even or odd) from $a_k$ to $z_n$. 
\end{enumerate}

Note that $\alpha$ is an arc visiting the points $z_1,z_2,\dots , z_n$ in order. Hence $Z$ is $n$-sac. The fact that $Z$ in not $(n+1)$-sac follows by
Lemma \ref{finite_cut_not_4sac}.
\end{proof}

Observe that for $n\ge 2$, Lemma \ref{n_trix} implies the existence of a $(2n-1)$-sac regular continuum $G$. Hence by Lemma \ref{finite_gluing} applied to
two disjoint copies of $G$, there exists an $n$-sac regular continuum that is not $(n+1)$-sac. We summarize this in the following theorem.

\begin{theorem}
For every $n\ge 2$ there exists a $n$-sac regular continuum that is not $(n+1)$-sac.
\end{theorem}

\section{Rational Curves}

In this section we construct an $\omega$-sac rational continuum.   The motivating idea for the construction is as follows. Recall that the closed disk is $3$-sac but not $4$-sac, and that this is because any arc visiting the cardinal points \emph{North}, \emph{South}, and then \emph{East} (while avoiding \emph{West}) cuts the disk in two, and can not be extended on to \emph{West}. Evidently if we add a handle to the disk, we can use this handle as a bridge from \emph{East} to \emph{West}, and this new space is $4$-sac. (In fact a closed disk with a handle is $6$-sac but not $7$-sac.) If we add more and more handles then the derived space will be $n$-sac for larger and larger $n$. So to construct our example of an 
$\omega$-sac rational continuum we start by modifying Charatonik's description of an example due to Urysohn. This is a rational continuum such that removing no finite subset disconnects it --- and so does not fail to be $\omega$-sac for that reason. However, like the closed disk, it is not $4$-sac for another reason: it is planar. So we further modify the space by adding an infinite and dense set of handles. This has be done so as to preserve rationality.

\begin{theorem} \label{omega_rational}
There is a locally connected rational continuum which is $\omega$-sac.
\end{theorem}
\begin{proof}
Write $B(\mathbf{x} , r)$ for the open disk in the plane of radius $r$ centered at $\mathbf{x}$. Write $S(\mathbf{x}, r)$ for the boundary circle of $B(\mathbf{x},r)$. Pick any monotone sequence $(x_n)_{n=0}^\infty$ in $(0,1)$ increasing to $1$. Let $c_0=0$, $r_0=x_0$ and $c_n=(x_n+x_{n-1})/2$, $r_n=(x_n-x_{n-1})/2$ for $n \ge 1$. Let $\theta$ be rotation of the plane by $90^o$ clockwise. Let $U = \bigcup_{i=0}^3 \theta^i \left( \bigcup_{n=0}^\infty B( (c_n,0), r_n)  \right)$. 

Let $T = [-1,+1]^2 \setminus U$. Let $S$ be the geometric boundary of $T$, so $S= \bigcup_{i=0}^3 \theta^i \left( \bigcup_{n=0}^\infty S( (c_n,0), r_n)  \right) \cup \partial [-1,1]^2$. Let $S^- = S((-c_1,0), r_1)$ and $S^+=S((c_1,0),r_1)$ (the two circles in $S$ immediately to the left and right of the center circle). For i=0 (respectively, i=1) pick a two sided sequence of points, $(p_{m,i}^-)_{m \in \mathbb{Z}\cup \{\pm \infty\}}$, on the top (respectively, bottom) edge of $S^-$ converging on the left to $p_{-\infty,i}^-=(-x_1-r_1,0)$ (the leftmost point of $S^-$) and on the right to $p_{\infty,i}^-=(-x_1+r_1,0)$ (the rightmost point of $S^-$). Find a corresponding pair, $(p_{m,i}^+)_{m \in \mathbb{Z}\cup \{\pm \infty\}}$ for $i=0,$ and~$1$  of double sequences on the top and bottom edges of $S^+$ converging to the leftmost and rightmost points of $S^+$.
Let $T_i = \theta^i \left( [0,1]^2 \cap T \right)$ for $i=0,1,2,3$. Note that each $T_i$ is a topological rectangle with natural `corners' and `midpoints' of the sides. 

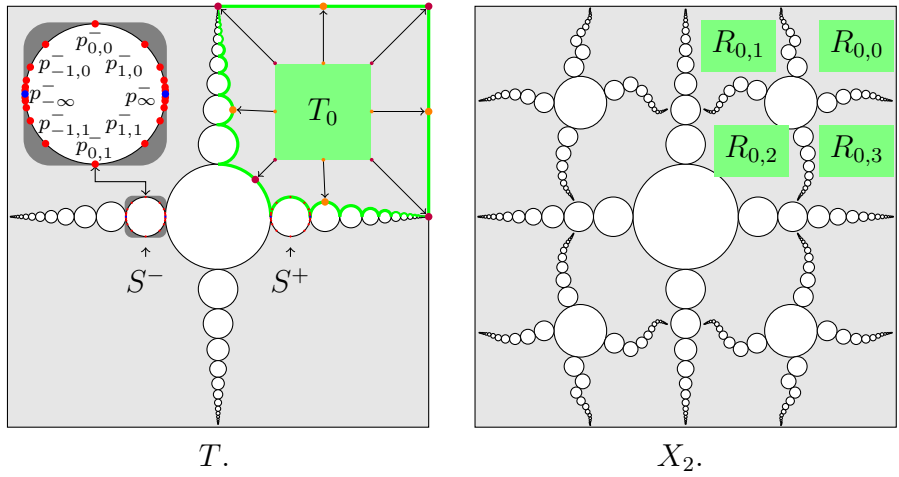
\begin{center}
 \begin{tabular}{cc}
\beginpgfgraphicnamed{fig_BaseT}
\begin{tikzpicture}[scale=2.8]
\draw[fill=gray!20] (-1,-1) rectangle (1,1);

\node (bgz) [shape=rectangle,rounded corners=10pt,fill=gray, minimum width=0.75in,minimum height=0.75in] at (-1/2-1/12,1/2+1/12) {};
\draw[fill=white] (-1/2-1/12,1/2+1/12) circle (1/3);
\node [shape=rectangle,rounded corners=3pt,fill=gray, minimum width=0.22in,minimum height=0.22in] (bgc) at (-0.34375 ,0) {};
\draw [<->] (bgc.north) |- (-7/12,2/12) -- (bgz.south);

\draw[fill=white] (0,0) circle (0.25);

\draw[green,very thick] ( 0.25, 0 ) arc (0:90: 0.25);
\draw[green,very thick] (0,1) -- (1,1) -- (1,0);

\draw[fill=white]  ( 0.34375 ,0) circle ( 0.09375 );
\node (Splus) at (0.34375, -0.09375) {};
\draw[fill=white] ( 0.5078125 ,0) circle ( 0.0703125 );
\draw[fill=white] ( 0.630859375 ,0) circle ( 0.052734375 );
\draw[fill=white] ( 0.72314453125 ,0) circle ( 0.03955078125 );
\draw[fill=white] ( 0.792358398438 ,0) circle ( 0.0296630859375 );
\draw[fill=white] ( 0.844268798828 ,0) circle ( 0.0222473144531 );
\draw[fill=white] ( 0.883201599121 ,0) circle ( 0.0166854858398 );
\draw[fill=white] ( 0.912401199341 ,0) circle ( 0.0125141143799 );
\draw[fill=white] ( 0.934300899506 ,0) circle ( 0.00938558578491 );
\draw[fill=white] ( 0.950725674629 ,0) circle ( 0.00703918933868 );
\draw[fill=white] ( 0.963044255972 ,0) circle ( 0.00527939200401 );
\draw[fill=white] ( 0.972283191979 ,0) circle ( 0.00395954400301 );
\draw[fill=white] ( 0.979212393984 ,0) circle ( 0.00296965800226 );
\draw[fill=white] ( 0.984409295488 ,0) circle ( 0.00222724350169 );

\draw[green,very thick] ( 0.25 ,0) arc (180:0: 0.09375 ) -- ( 0.4375 ,0) arc (180:0: 0.0703125 );
\draw[green,very thick] ( 0.578125 ,0) arc (180:0: 0.052734375 );
\draw[green,very thick] ( 0.68359375 ,0) arc (180:0: 0.03955078125 );
\draw[green,very thick] ( 0.7626953125 ,0) arc (180:0: 0.0296630859375 );
\draw[green,thick] ( 0.822021484375 ,0) arc (180:0: 0.0222473144531 );
\draw[green,thick] ( 0.866516113281 ,0) arc (180:0: 0.0166854858398 );
\draw[green,thick] ( 0.899887084961 ,0) arc (180:0: 0.0125141143799 );
\draw[green,thick] ( 0.924915313721 ,0) arc (180:0: 0.00938558578491 );
\draw[green,thick] ( 0.943686485291 ,0) arc (180:0: 0.00703918933868 );
\draw[green] ( 0.957764863968 ,0) arc (180:0: 0.00527939200401 );
\draw[green] ( 0.968323647976 ,0) arc (180:0: 0.00395954400301 );
\draw[green] ( 0.976242735982 ,0) arc (180:0: 0.00296965800226 );
\draw[green] ( 0.982182051986 ,0) arc (180:0: 0.00222724350169 ) -- (1,0);

\draw[fill=white] ( -0.34375 ,0) circle ( 0.09375 );
\node (Sminus) at (-0.34375, -0.09375) {};
\draw[fill=white] ( -0.5078125 ,0) circle ( 0.0703125 );
\draw[fill=white] ( -0.630859375 ,0) circle ( 0.052734375 );
\draw[fill=white] ( -0.72314453125 ,0) circle ( 0.03955078125 );
\draw[fill=white] ( -0.792358398438 ,0) circle ( 0.0296630859375 );
\draw[fill=white] ( -0.844268798828 ,0) circle ( 0.0222473144531 );
\draw[fill=white] ( -0.883201599121 ,0) circle ( 0.0166854858398 );
\draw[fill=white] ( -0.912401199341 ,0) circle ( 0.0125141143799 );
\draw[fill=white] ( -0.934300899506 ,0) circle ( 0.00938558578491 );
\draw[fill=white] ( -0.950725674629 ,0) circle ( 0.00703918933868 );
\draw[fill=white] ( -0.963044255972 ,0) circle ( 0.00527939200401 );
\draw[fill=white] ( -0.972283191979 ,0) circle ( 0.00395954400301 );
\draw[fill=white] ( -0.979212393984 ,0) circle ( 0.00296965800226 );
\draw[fill=white] ( -0.984409295488 ,0) circle ( 0.00222724350169 );

\draw[fill=white] (0, 0.34375 ) circle ( 0.09375 );
\draw[fill=white] (0, 0.5078125 ) circle ( 0.0703125 );
\draw[fill=white] (0, 0.630859375 ) circle ( 0.052734375 );
\draw[fill=white] (0, 0.72314453125 ) circle ( 0.03955078125 );
\draw[fill=white] (0, 0.792358398438 ) circle ( 0.0296630859375 );
\draw[fill=white] (0, 0.844268798828 ) circle ( 0.0222473144531 );
\draw[fill=white] (0, 0.883201599121 ) circle ( 0.0166854858398 );
\draw[fill=white] (0, 0.912401199341 ) circle ( 0.0125141143799 );
\draw[fill=white] (0, 0.934300899506 ) circle ( 0.00938558578491 );
\draw[fill=white] (0, 0.950725674629 ) circle ( 0.00703918933868 );
\draw[fill=white] (0, 0.963044255972 ) circle ( 0.00527939200401 );
\draw[fill=white] (0, 0.972283191979 ) circle ( 0.00395954400301 );
\draw[fill=white] (0, 0.979212393984 ) circle ( 0.00296965800226 );
\draw[fill=white] (0, 0.984409295488 ) circle ( 0.00222724350169 );

\draw[green,very thick] ( 0, 0.25) arc (-90:90: 0.09375 ) -- ( 0, 0.4375) arc (-90:90: 0.0703125 );
\draw[green,very thick] ( 0, 0.578125 ) arc (-90:90: 0.052734375 );
\draw[green,very thick] ( 0, 0.68359375) arc (-90:90: 0.03955078125 );
\draw[green,very thick] ( 0, 0.7626953125) arc (-90:90: 0.0296630859375 );
\draw[green,thick] ( 0, 0.822021484375) arc (-90:90: 0.0222473144531 );
\draw[green,thick] ( 0, 0.866516113281) arc (-90:90: 0.0166854858398 );
\draw[green,thick] ( 0, 0.899887084961) arc (-90:90: 0.0125141143799 );
\draw[green,thick] ( 0, 0.924915313721) arc (-90:90: 0.00938558578491 );
\draw[green,thick] ( 0, 0.943686485291) arc (-90:90: 0.00703918933868 );
\draw[green] ( 0, 0.957764863968) arc (-90:90: 0.00527939200401 );
\draw[green] ( 0, 0.968323647976) arc (-90:90: 0.00395954400301 );
\draw[green] ( 0, 0.976242735982) arc (-90:90: 0.00296965800226 );
\draw[green] ( 0, 0.982182051986) arc (-90:90: 0.00222724350169 ) -- (0,1);

\draw[fill=white] (0, -0.34375 ) circle ( 0.09375 );
\draw[fill=white] (0, -0.5078125 ) circle ( 0.0703125 );
\draw[fill=white] (0, -0.630859375 ) circle ( 0.052734375 );
\draw[fill=white] (0, -0.72314453125 ) circle ( 0.03955078125 );
\draw[fill=white] (0, -0.792358398438 ) circle ( 0.0296630859375 );
\draw[fill=white] (0, -0.844268798828 ) circle ( 0.0222473144531 );
\draw[fill=white] (0, -0.883201599121 ) circle ( 0.0166854858398 );
\draw[fill=white] (0, -0.912401199341 ) circle ( 0.0125141143799 );
\draw[fill=white] (0, -0.934300899506 ) circle ( 0.00938558578491 );
\draw[fill=white] (0, -0.950725674629 ) circle ( 0.00703918933868 );
\draw[fill=white] (0, -0.963044255972 ) circle ( 0.00527939200401 );
\draw[fill=white] (0, -0.972283191979 ) circle ( 0.00395954400301 );
\draw[fill=white] (0, -0.979212393984 ) circle ( 0.00296965800226 );
\draw[fill=white] (0, -0.984409295488 ) circle ( 0.00222724350169 );

\node (Sm) at (-0.34375,-0.3) {$S^-$};
\draw [->, thin] (Sm.north) -- (Sminus.south);
\node (Sp) at (0.34375,-0.3) {$S^+$};
\draw [->, thin] (Sp.north) -- (Splus.south);

\foreach \an in {-45/16,-45/8,-45/4,-45/2,-45,90,45,45/2,45/4,45/8, 45/16}
    \fill[red] ($(0.34375,0) +((\an:0.09375)$) circle (0.15pt);
\foreach \an in {180-45/16,180-45/8,180-45/4,180-45/2,180-45,270,180+45,180+45/2,180+45/4,180+45/8,180+45/16}
    \fill[red] ($(0.34375,0) +((\an:0.09375)$) circle (0.15pt);
\fill[blue] ($(0.34375,0) +((0:0.09375)$) circle (0.15pt);
\fill[blue] ($(0.34375,0) +((180:0.09375)$) circle (0.15pt);

\foreach \an in {-45/16,-45/8,-45/4,-45/2,-45,90,45,45/2,45/4,45/8, 45/16}
    \fill[red] ($(-0.34375,0) +((\an:0.09375)$) circle (0.15pt);
\foreach \an in {180-45/16,180-45/8,180-45/4,180-45/2,180-45,270,180+45,180+45/2,180+45/4,180+45/8,180+45/16}
    \fill[red] ($(-0.34375,0) +((\an:0.09375)$) circle (0.15pt);
\fill[blue] ($(-0.34375,0) +((0:0.09375)$) circle (0.15pt);
\fill[blue] ($(-0.34375,0) +((180:0.09375)$) circle (0.15pt);

\foreach \an in {-45/16,-45/8,-45/4,-45/2,-45,90,45,45/2,45/4,45/8, 45/16}
    \fill[red] ($(-1/2-1/12,1/2+1/12) +((\an:0.3333333)$) circle (0.5pt);
\node () at ($(-1/2-1/12,1/2+1/12) +((90:0.25)$) {\scriptsize{$p_{0,0}^-$}};
\node () at ($(-1/2-1/12,1/2+1/12) +((45:0.2)$) {\scriptsize{$p_{1,0}^-$}};
\node () at ($(-1/2-1/12,1/2+1/12) +((135:0.2)$) {\scriptsize{$p_{-1,0}^-$}};

\foreach \an in {180-45/16,180-45/8,180-45/4,180-45/2,180-45,270,180+45,180+45/2,180+45/4,180+45/8,180+45/16}
    \fill[red] ($(-1/2-1/12,1/2+1/12) +((\an:0.3333333)$) circle (0.5pt);
\node () at ($(-1/2-1/12,1/2+1/12) +((-90:0.25)$) {\scriptsize{$p_{0,1}^-$}};
\node () at ($(-1/2-1/12,1/2+1/12) +((-45:0.2)$) {\scriptsize{$p_{1,1}^-$}};
\node () at ($(-1/2-1/12,1/2+1/12) +((-135:0.2)$) {\scriptsize{$p_{-1,1}^-$}};

\fill[blue] ($(-1/2-1/12,1/2+1/12) +((0:0.3333333)$) circle (0.5pt);
\node () at ($(-1/2-1/12,1/2+1/12) +((0:0.22)$) {\scriptsize{$p_{\infty}^-$}};

\fill[blue] ($(-1/2-1/12,1/2+1/12) +((180:0.3333333)$) circle (0.5pt);
\node () at ($(-1/2-1/12,1/2+1/12) +((180:0.2)$) {\scriptsize{$p_{-\infty}^-$}};

\node[fill=green!50, minimum width=0.5in, minimum height=0.5in] (T0m) at (1/2,1/2) {$T_0$};
\node [circle,fill=purple,inner sep=1pt] (c1) at (45:0.25) {};
\draw [very thin, -> ] (T0m.south west) -- (c1.north east);
\node [circle,fill=purple,inner sep=0.5pt]  at (T0m.south west) {};

\node [circle,fill=purple,inner sep=1pt] (c2) at (0,1) {};
\draw [very thin, -> ] (T0m.north west) -- (c2.south east);
\node [circle,fill=purple,inner sep=0.5pt]  at (T0m.north west) {};

\node [circle,fill=purple,inner sep=1pt] (c3) at (1,1) {};
\draw [very thin, ->] (T0m.north east) -- (c3.south west);
\node [circle,fill=purple,inner sep=0.5pt]  at (T0m.north east) {};

\node [circle,fill=purple,inner sep=1pt] (c4) at (1,0) {};
\draw [very thin, ->] (T0m.south east) -- (c4.north west);
\node [circle,fill=purple,inner sep=0.5pt]  at (T0m.south east) {};

\node [circle,fill=orange,inner sep=1pt] (m1) at (0.5078125,0.0703125) {};
\draw [very thin, -> ] (T0m.south) -- (m1.north);
\node [circle,fill=orange,inner sep=0.5pt]  at (T0m.south) {};

\node [circle,fill=orange,inner sep=1pt] (m2) at (0.0703125,0.5078125)  {};
\draw [very thin, -> ] (T0m.west) -- (m2.east);
\node [circle,fill=orange,inner sep=0.5pt]  at (T0m.west) {};

\node [circle,fill=orange,inner sep=1pt] (m3) at (1/2,1) {};
\draw [very thin, ->] (T0m.north) -- (m3.south);
\node [circle,fill=orange,inner sep=0.5pt]  at (T0m.north) {};

\node [circle,fill=orange,inner sep=1pt] (m4) at (1,1/2) {};
\draw [very thin, ->] (T0m.east) -- (m4.west);
\node [circle,fill=orange,inner sep=0.5pt]  at (T0m.east) {};

\end{tikzpicture}
\endpgfgraphicnamed
&
\beginpgfgraphicnamed{fig_X2}
\begin{tikzpicture}[scale=2.8]
\draw[fill=gray!20] (-1,-1) rectangle (1,1);

\draw[fill=white] (0,0) circle (0.25);

\draw[fill=white]  ( 0.34375 ,0) circle ( 0.09375 );
\node (Splus) at (0.34375, -0.09375) {};
\draw[fill=white] ( 0.5078125 ,0) circle ( 0.0703125 );
\draw[fill=white] ( 0.630859375 ,0) circle ( 0.052734375 );
\draw[fill=white] ( 0.72314453125 ,0) circle ( 0.03955078125 );
\draw[fill=white] ( 0.792358398438 ,0) circle ( 0.0296630859375 );
\draw[fill=white] ( 0.844268798828 ,0) circle ( 0.0222473144531 );
\draw[fill=white] ( 0.883201599121 ,0) circle ( 0.0166854858398 );
\draw[fill=white] ( 0.912401199341 ,0) circle ( 0.0125141143799 );
\draw[fill=white] ( 0.934300899506 ,0) circle ( 0.00938558578491 );
\draw[fill=white] ( 0.950725674629 ,0) circle ( 0.00703918933868 );
\draw[fill=white] ( 0.963044255972 ,0) circle ( 0.00527939200401 );
\draw[fill=white] ( 0.972283191979 ,0) circle ( 0.00395954400301 );
\draw[fill=white] ( 0.979212393984 ,0) circle ( 0.00296965800226 );
\draw[fill=white] ( 0.984409295488 ,0) circle ( 0.00222724350169 );

\draw[fill=white] ( -0.34375 ,0) circle ( 0.09375 );
\node (Sminus) at (-0.34375, -0.09375) {};
\draw[fill=white] ( -0.5078125 ,0) circle ( 0.0703125 );
\draw[fill=white] ( -0.630859375 ,0) circle ( 0.052734375 );
\draw[fill=white] ( -0.72314453125 ,0) circle ( 0.03955078125 );
\draw[fill=white] ( -0.792358398438 ,0) circle ( 0.0296630859375 );
\draw[fill=white] ( -0.844268798828 ,0) circle ( 0.0222473144531 );
\draw[fill=white] ( -0.883201599121 ,0) circle ( 0.0166854858398 );
\draw[fill=white] ( -0.912401199341 ,0) circle ( 0.0125141143799 );
\draw[fill=white] ( -0.934300899506 ,0) circle ( 0.00938558578491 );
\draw[fill=white] ( -0.950725674629 ,0) circle ( 0.00703918933868 );
\draw[fill=white] ( -0.963044255972 ,0) circle ( 0.00527939200401 );
\draw[fill=white] ( -0.972283191979 ,0) circle ( 0.00395954400301 );
\draw[fill=white] ( -0.979212393984 ,0) circle ( 0.00296965800226 );
\draw[fill=white] ( -0.984409295488 ,0) circle ( 0.00222724350169 );

\draw[fill=white] (0, 0.34375 ) circle ( 0.09375 );
\draw[fill=white] (0, 0.5078125 ) circle ( 0.0703125 );
\draw[fill=white] (0, 0.630859375 ) circle ( 0.052734375 );
\draw[fill=white] (0, 0.72314453125 ) circle ( 0.03955078125 );
\draw[fill=white] (0, 0.792358398438 ) circle ( 0.0296630859375 );
\draw[fill=white] (0, 0.844268798828 ) circle ( 0.0222473144531 );
\draw[fill=white] (0, 0.883201599121 ) circle ( 0.0166854858398 );
\draw[fill=white] (0, 0.912401199341 ) circle ( 0.0125141143799 );
\draw[fill=white] (0, 0.934300899506 ) circle ( 0.00938558578491 );
\draw[fill=white] (0, 0.950725674629 ) circle ( 0.00703918933868 );
\draw[fill=white] (0, 0.963044255972 ) circle ( 0.00527939200401 );
\draw[fill=white] (0, 0.972283191979 ) circle ( 0.00395954400301 );
\draw[fill=white] (0, 0.979212393984 ) circle ( 0.00296965800226 );
\draw[fill=white] (0, 0.984409295488 ) circle ( 0.00222724350169 );

\draw[fill=white] (0, -0.34375 ) circle ( 0.09375 );
\draw[fill=white] (0, -0.5078125 ) circle ( 0.0703125 );
\draw[fill=white] (0, -0.630859375 ) circle ( 0.052734375 );
\draw[fill=white] (0, -0.72314453125 ) circle ( 0.03955078125 );
\draw[fill=white] (0, -0.792358398438 ) circle ( 0.0296630859375 );
\draw[fill=white] (0, -0.844268798828 ) circle ( 0.0222473144531 );
\draw[fill=white] (0, -0.883201599121 ) circle ( 0.0166854858398 );
\draw[fill=white] (0, -0.912401199341 ) circle ( 0.0125141143799 );
\draw[fill=white] (0, -0.934300899506 ) circle ( 0.00938558578491 );
\draw[fill=white] (0, -0.950725674629 ) circle ( 0.00703918933868 );
\draw[fill=white] (0, -0.963044255972 ) circle ( 0.00527939200401 );
\draw[fill=white] (0, -0.972283191979 ) circle ( 0.00395954400301 );
\draw[fill=white] (0, -0.979212393984 ) circle ( 0.00296965800226 );
\draw[fill=white] (0, -0.984409295488 ) circle ( 0.00222724350169 );

\draw[fill=white] ( 0.5 , 0.54202 ) circle ( 0.125 );
\draw[fill=white] ( 0.558784712134 , 0.380510330802 ) circle ( 0.046875 );
\draw[fill=white] ( 0.580015961928 , 0.301274227865 ) circle ( 0.03515625 );
\draw[fill=white] ( 0.585378082819 , 0.239984905619 ) circle ( 0.0263671875 );
\draw[fill=white] ( 0.585378082819 , 0.193842327494 ) circle ( 0.019775390625 );
\draw[fill=white] ( 0.579368651866 , 0.159761150983 ) circle ( 0.0148315429688 );
\draw[fill=white] ( 0.568399510276 , 0.136237750932 ) circle ( 0.0111236572266 );
\draw[fill=white] ( 0.558666310203 , 0.119379353884 ) circle ( 0.00834274291992 );
\draw[fill=white] ( 0.550292208885 , 0.107419897778 ) circle ( 0.00625705718994 );
\draw[fill=white] ( 0.543253780924 , 0.0990318259697 ) circle ( 0.00469279289246 );
\draw[fill=white] ( 0.539147587143 , 0.0919196897154 ) circle ( 0.00351959466934 );
\draw[fill=white] ( 0.536067941807 , 0.0865855875247 ) circle ( 0.00263969600201 );
\draw[fill=white] ( 0.533758207806 , 0.0825850108817 ) circle ( 0.0019797720015 );
\draw[fill=white] ( 0.532025907304 , 0.0795845783995 ) circle ( 0.00148482900113 ); 

\draw[fill=white] ( 0.671875 , 0.54202 ) circle ( 0.046875 );
\draw[fill=white] ( 0.748959160299 , 0.570076339882 ) circle ( 0.03515625 );
\draw[fill=white] ( 0.809547918541 , 0.580759772688 ) circle ( 0.0263671875 );
\draw[fill=white] ( 0.855690496666 , 0.580759772688 ) circle ( 0.019775390625 );
\draw[fill=white] ( 0.889771673177 , 0.574750341735 ) circle ( 0.0148315429688 );
\draw[fill=white] ( 0.914161583271 , 0.565873140444 ) circle ( 0.0111236572266 );
\draw[fill=white] ( 0.932454015842 , 0.559215239476 ) circle ( 0.00834274291992 );
\draw[fill=white] ( 0.946832012183 , 0.556680010792 ) circle ( 0.00625705718994 );
\draw[fill=white] ( 0.957615509438 , 0.55477858928 ) circle ( 0.00469279289246 );
\draw[fill=white] ( 0.96570313238 , 0.553352523145 ) circle ( 0.00351959466934 );
\draw[fill=white] ( 0.971768849586 , 0.552282973545 ) circle ( 0.00263969600201 );
\draw[fill=white] ( 0.976318137491 , 0.551480811344 ) circle ( 0.0019797720015 );
\draw[fill=white] ( 0.979730103419 , 0.550879189694 ) circle ( 0.00148482900113 );

\draw[fill=white] ( 0.338490330802 , 0.600804712134 ) circle ( 0.046875 );
\draw[fill=white] ( 0.261406170504 , 0.628861052016 ) circle ( 0.03515625 );
\draw[fill=white] ( 0.2081253107 , 0.598099333266 ) circle ( 0.0263671875 );
\draw[fill=white] ( 0.178465433203 , 0.562752067702 ) circle ( 0.019775390625 );
\draw[fill=white] ( 0.161161966406 , 0.532781584063 ) circle ( 0.0148315429688 );
\draw[fill=white] ( 0.148184366308 , 0.510303721334 ) circle ( 0.0111236572266 );
\draw[fill=white] ( 0.133272238648 , 0.497790960514 ) circle ( 0.00834274291992 );
\draw[fill=white] ( 0.118672438539 , 0.497790960514 ) circle ( 0.00625705718994 );
\draw[fill=white] ( 0.108382945217 , 0.501536029809 ) circle ( 0.00469279289246 );
\draw[fill=white] ( 0.100665825227 , 0.50434483178 ) circle ( 0.00351959466934 );
\draw[fill=white] ( 0.0953317230359 , 0.507424477116 ) circle ( 0.00263969600201 );
\draw[fill=white] ( 0.0913311463929 , 0.509734211117 ) circle ( 0.0019797720015 );
\draw[fill=white] ( 0.0879191804644 , 0.510335832768 ) circle ( 0.00148482900113 );

\draw[fill=white] ( 0.529845780537 , 0.711283832549 ) circle ( 0.046875 );
\draw[fill=white] ( 0.488830155537 , 0.782324978953 ) circle ( 0.03515625 );
\draw[fill=white] ( 0.519591874287 , 0.835605838756 ) circle ( 0.0263671875 );
\draw[fill=white] ( 0.496520585224 , 0.875566483609 ) circle ( 0.019775390625 );
\draw[fill=white] ( 0.474275677101 , 0.902076932782 ) circle ( 0.0148315429688 );
\draw[fill=white] ( 0.46539847581 , 0.926466842876 ) circle ( 0.0111236572266 );
\draw[fill=white] ( 0.455665275737 , 0.943325239923 ) circle ( 0.00834274291992 );
\draw[fill=white] ( 0.453130047053 , 0.957703236264 ) circle ( 0.00625705718994 );
\draw[fill=white] ( 0.453130047053 , 0.968653086346 ) circle ( 0.00469279289246 );
\draw[fill=white] ( 0.454556113188 , 0.976740709288 ) circle ( 0.00351959466934 );
\draw[fill=white] ( 0.454556113188 , 0.982899999959 ) circle ( 0.00263969600201 );
\draw[fill=white] ( 0.453753950987 , 0.987449287864 ) circle ( 0.0019797720015 );
\draw[fill=white] ( 0.452568987656 , 0.99070494786 ) circle ( 0.00148482900113 );

\draw[fill=white] ( -0.5 , 0.54202 ) circle ( 0.125 );
\draw[fill=white] ( -0.558784712134 , 0.380510330802 ) circle ( 0.046875 );
\draw[fill=white] ( -0.580015961928 , 0.301274227865 ) circle ( 0.03515625 );
\draw[fill=white] ( -0.585378082819 , 0.239984905619 ) circle ( 0.0263671875 );
\draw[fill=white] ( -0.585378082819 , 0.193842327494 ) circle ( 0.019775390625 );
\draw[fill=white] ( -0.579368651866 , 0.159761150983 ) circle ( 0.0148315429688 );
\draw[fill=white] ( -0.568399510276 , 0.136237750932 ) circle ( 0.0111236572266 );
\draw[fill=white] ( -0.558666310203 , 0.119379353884 ) circle ( 0.00834274291992 );
\draw[fill=white] ( -0.550292208885 , 0.107419897778 ) circle ( 0.00625705718994 );
\draw[fill=white] ( -0.543253780924 , 0.0990318259697 ) circle ( 0.00469279289246 );
\draw[fill=white] ( -0.539147587143 , 0.0919196897154 ) circle ( 0.00351959466934 );
\draw[fill=white] ( -0.536067941807 , 0.0865855875247 ) circle ( 0.00263969600201 );
\draw[fill=white] ( -0.533758207806 , 0.0825850108817 ) circle ( 0.0019797720015 );
\draw[fill=white] ( -0.532025907304 , 0.0795845783995 ) circle ( 0.00148482900113 );

\draw[fill=white] ( -0.671875 , 0.54202 ) circle ( 0.046875 );
\draw[fill=white] ( -0.748959160299 , 0.570076339882 ) circle ( 0.03515625 );
\draw[fill=white] ( -0.809547918541 , 0.580759772688 ) circle ( 0.0263671875 );
\draw[fill=white] ( -0.855690496666 , 0.580759772688 ) circle ( 0.019775390625 );
\draw[fill=white] ( -0.889771673177 , 0.574750341735 ) circle ( 0.0148315429688 );
\draw[fill=white] ( -0.914161583271 , 0.565873140444 ) circle ( 0.0111236572266 );
\draw[fill=white] ( -0.932454015842 , 0.559215239476 ) circle ( 0.00834274291992 );
\draw[fill=white] ( -0.946832012183 , 0.556680010792 ) circle ( 0.00625705718994 );
\draw[fill=white] ( -0.957615509438 , 0.55477858928 ) circle ( 0.00469279289246 );
\draw[fill=white] ( -0.96570313238 , 0.553352523145 ) circle ( 0.00351959466934 );
\draw[fill=white] ( -0.971768849586 , 0.552282973545 ) circle ( 0.00263969600201 );
\draw[fill=white] ( -0.976318137491 , 0.551480811344 ) circle ( 0.0019797720015 );
\draw[fill=white] ( -0.979730103419 , 0.550879189694 ) circle ( 0.00148482900113 );

\draw[fill=white] ( -0.338490330802 , 0.600804712134 ) circle ( 0.046875 );
\draw[fill=white] ( -0.261406170504 , 0.628861052016 ) circle ( 0.03515625 );
\draw[fill=white] ( -0.2081253107 , 0.598099333266 ) circle ( 0.0263671875 );
\draw[fill=white] ( -0.178465433203 , 0.562752067702 ) circle ( 0.019775390625 );
\draw[fill=white] ( -0.161161966406 , 0.532781584063 ) circle ( 0.0148315429688 );
\draw[fill=white] ( -0.148184366308 , 0.510303721334 ) circle ( 0.0111236572266 );
\draw[fill=white] ( -0.133272238648 , 0.497790960514 ) circle ( 0.00834274291992 );
\draw[fill=white] ( -0.118672438539 , 0.497790960514 ) circle ( 0.00625705718994 );
\draw[fill=white] ( -0.108382945217 , 0.501536029809 ) circle ( 0.00469279289246 );
\draw[fill=white] ( -0.100665825227 , 0.50434483178 ) circle ( 0.00351959466934 );
\draw[fill=white] ( -0.0953317230359 , 0.507424477116 ) circle ( 0.00263969600201 );
\draw[fill=white] ( -0.0913311463929 , 0.509734211117 ) circle ( 0.0019797720015 );
\draw[fill=white] ( -0.0879191804644 , 0.510335832768 ) circle ( 0.00148482900113 );

\draw[fill=white] ( -0.529845780537 , 0.711283832549 ) circle ( 0.046875 );
\draw[fill=white] ( -0.488830155537 , 0.782324978953 ) circle ( 0.03515625 );
\draw[fill=white] ( -0.519591874287 , 0.835605838756 ) circle ( 0.0263671875 );
\draw[fill=white] ( -0.496520585224 , 0.875566483609 ) circle ( 0.019775390625 );
\draw[fill=white] ( -0.474275677101 , 0.902076932782 ) circle ( 0.0148315429688 );
\draw[fill=white] ( -0.46539847581 , 0.926466842876 ) circle ( 0.0111236572266 );
\draw[fill=white] ( -0.455665275737 , 0.943325239923 ) circle ( 0.00834274291992 );
\draw[fill=white] ( -0.453130047053 , 0.957703236264 ) circle ( 0.00625705718994 );
\draw[fill=white] ( -0.453130047053 , 0.968653086346 ) circle ( 0.00469279289246 );
\draw[fill=white] ( -0.454556113188 , 0.976740709288 ) circle ( 0.00351959466934 );
\draw[fill=white] ( -0.454556113188 , 0.982899999959 ) circle ( 0.00263969600201 );
\draw[fill=white] ( -0.453753950987 , 0.987449287864 ) circle ( 0.0019797720015 );
\draw[fill=white] ( -0.452568987656 , 0.99070494786 ) circle ( 0.00148482900113 );

\node[fill=green!50] () at (3/4+1/16,3/4+1/16) {$R_{0,0}$};
\node[fill=green!50] () at (3/4+1/16,1/4+1/16) {$R_{0,3}$};
\node[fill=green!50] () at (1/4,3/4+1/16) {$R_{0,1}$};
\node[fill=green!50] () at (1/4+1/16,1/4+1/16) {$R_{0,2}$};

\draw[fill=white] ( -0.5 , -0.54202 ) circle ( 0.125 );
\draw[fill=white] ( -0.558784712134 , -0.380510330802 ) circle ( 0.046875 );
\draw[fill=white] ( -0.580015961928 , -0.301274227865 ) circle ( 0.03515625 );
\draw[fill=white] ( -0.585378082819 , -0.239984905619 ) circle ( 0.0263671875 );
\draw[fill=white] ( -0.585378082819 , -0.193842327494 ) circle ( 0.019775390625 );
\draw[fill=white] ( -0.579368651866 , -0.159761150983 ) circle ( 0.0148315429688 );
\draw[fill=white] ( -0.568399510276 , -0.136237750932 ) circle ( 0.0111236572266 );
\draw[fill=white] ( -0.558666310203 , -0.119379353884 ) circle ( 0.00834274291992 );
\draw[fill=white] ( -0.550292208885 , -0.107419897778 ) circle ( 0.00625705718994 );
\draw[fill=white] ( -0.543253780924 , -0.0990318259697 ) circle ( 0.00469279289246 );
\draw[fill=white] ( -0.539147587143 , -0.0919196897154 ) circle ( 0.00351959466934 );
\draw[fill=white] ( -0.536067941807 , -0.0865855875247 ) circle ( 0.00263969600201 );
\draw[fill=white] ( -0.533758207806 , -0.0825850108817 ) circle ( 0.0019797720015 );
\draw[fill=white] ( -0.532025907304 , -0.0795845783995 ) circle ( 0.00148482900113 );

\draw[fill=white] ( -0.671875 , -0.54202 ) circle ( 0.046875 );
\draw[fill=white] ( -0.748959160299 , -0.570076339882 ) circle ( 0.03515625 );
\draw[fill=white] ( -0.809547918541 , -0.580759772688 ) circle ( 0.0263671875 );
\draw[fill=white] ( -0.855690496666 , -0.580759772688 ) circle ( 0.019775390625 );
\draw[fill=white] ( -0.889771673177 , -0.574750341735 ) circle ( 0.0148315429688 );
\draw[fill=white] ( -0.914161583271 , -0.565873140444 ) circle ( 0.0111236572266 );
\draw[fill=white] ( -0.932454015842 , -0.559215239476 ) circle ( 0.00834274291992 );
\draw[fill=white] ( -0.946832012183 , -0.556680010792 ) circle ( 0.00625705718994 );
\draw[fill=white] ( -0.957615509438 , -0.55477858928 ) circle ( 0.00469279289246 );
\draw[fill=white] ( -0.96570313238 , -0.553352523145 ) circle ( 0.00351959466934 );
\draw[fill=white] ( -0.971768849586 , -0.552282973545 ) circle ( 0.00263969600201 );
\draw[fill=white] ( -0.976318137491 , -0.551480811344 ) circle ( 0.0019797720015 );
\draw[fill=white] ( -0.979730103419 , -0.550879189694 ) circle ( 0.00148482900113 );

\draw[fill=white] ( -0.338490330802 , -0.600804712134 ) circle ( 0.046875 );
\draw[fill=white] ( -0.261406170504 , -0.628861052016 ) circle ( 0.03515625 );
\draw[fill=white] ( -0.2081253107 , -0.598099333266 ) circle ( 0.0263671875 );
\draw[fill=white] ( -0.178465433203 , -0.562752067702 ) circle ( 0.019775390625 );
\draw[fill=white] ( -0.161161966406 , -0.532781584063 ) circle ( 0.0148315429688 );
\draw[fill=white] ( -0.148184366308 , -0.510303721334 ) circle ( 0.0111236572266 );
\draw[fill=white] ( -0.133272238648 , -0.497790960514 ) circle ( 0.00834274291992 );
\draw[fill=white] ( -0.118672438539 , -0.497790960514 ) circle ( 0.00625705718994 );
\draw[fill=white] ( -0.108382945217 , -0.501536029809 ) circle ( 0.00469279289246 );
\draw[fill=white] ( -0.100665825227 , -0.50434483178 ) circle ( 0.00351959466934 );
\draw[fill=white] ( -0.0953317230359 , -0.507424477116 ) circle ( 0.00263969600201 );
\draw[fill=white] ( -0.0913311463929 , -0.509734211117 ) circle ( 0.0019797720015 );
\draw[fill=white] ( -0.0879191804644 , -0.510335832768 ) circle ( 0.00148482900113 );

\draw[fill=white] ( -0.529845780537 , -0.711283832549 ) circle ( 0.046875 );
\draw[fill=white] ( -0.488830155537 , -0.782324978953 ) circle ( 0.03515625 );
\draw[fill=white] ( -0.519591874287 , -0.835605838756 ) circle ( 0.0263671875 );
\draw[fill=white] ( -0.496520585224 , -0.875566483609 ) circle ( 0.019775390625 );
\draw[fill=white] ( -0.474275677101 , -0.902076932782 ) circle ( 0.0148315429688 );
\draw[fill=white] ( -0.46539847581 , -0.926466842876 ) circle ( 0.0111236572266 );
\draw[fill=white] ( -0.455665275737 , -0.943325239923 ) circle ( 0.00834274291992 );
\draw[fill=white] ( -0.453130047053 , -0.957703236264 ) circle ( 0.00625705718994 );
\draw[fill=white] ( -0.453130047053 , -0.968653086346 ) circle ( 0.00469279289246 );
\draw[fill=white] ( -0.454556113188 , -0.976740709288 ) circle ( 0.00351959466934 );
\draw[fill=white] ( -0.454556113188 , -0.982899999959 ) circle ( 0.00263969600201 );
\draw[fill=white] ( -0.453753950987 , -0.987449287864 ) circle ( 0.0019797720015 );
\draw[fill=white] ( -0.452568987656 , -0.99070494786 ) circle ( 0.00148482900113 );

\draw[fill=white] ( 0.5 , -0.54202 ) circle ( 0.125 );
\draw[fill=white] ( 0.558784712134 , -0.380510330802 ) circle ( 0.046875 );
\draw[fill=white] ( 0.580015961928 , -0.301274227865 ) circle ( 0.03515625 );
\draw[fill=white] ( 0.585378082819 , -0.239984905619 ) circle ( 0.0263671875 );
\draw[fill=white] ( 0.585378082819 , -0.193842327494 ) circle ( 0.019775390625 );
\draw[fill=white] ( 0.579368651866 , -0.159761150983 ) circle ( 0.0148315429688 );
\draw[fill=white] ( 0.568399510276 , -0.136237750932 ) circle ( 0.0111236572266 );
\draw[fill=white] ( 0.558666310203 , -0.119379353884 ) circle ( 0.00834274291992 );
\draw[fill=white] ( 0.550292208885 , -0.107419897778 ) circle ( 0.00625705718994 );
\draw[fill=white] ( 0.543253780924 , -0.0990318259697 ) circle ( 0.00469279289246 );
\draw[fill=white] ( 0.539147587143 , -0.0919196897154 ) circle ( 0.00351959466934 );
\draw[fill=white] ( 0.536067941807 , -0.0865855875247 ) circle ( 0.00263969600201 );
\draw[fill=white] ( 0.533758207806 , -0.0825850108817 ) circle ( 0.0019797720015 );
\draw[fill=white] ( 0.532025907304 , -0.0795845783995 ) circle ( 0.00148482900113 ); 

\draw[fill=white] ( 0.671875 , -0.54202 ) circle ( 0.046875 );
\draw[fill=white] ( 0.748959160299 , -0.570076339882 ) circle ( 0.03515625 );
\draw[fill=white] ( 0.809547918541 , -0.580759772688 ) circle ( 0.0263671875 );
\draw[fill=white] ( 0.855690496666 , -0.580759772688 ) circle ( 0.019775390625 );
\draw[fill=white] ( 0.889771673177 , -0.574750341735 ) circle ( 0.0148315429688 );
\draw[fill=white] ( 0.914161583271 , -0.565873140444 ) circle ( 0.0111236572266 );
\draw[fill=white] ( 0.932454015842 , -0.559215239476 ) circle ( 0.00834274291992 );
\draw[fill=white] ( 0.946832012183 , -0.556680010792 ) circle ( 0.00625705718994 );
\draw[fill=white] ( 0.957615509438 , -0.55477858928 ) circle ( 0.00469279289246 );
\draw[fill=white] ( 0.96570313238 , -0.553352523145 ) circle ( 0.00351959466934 );
\draw[fill=white] ( 0.971768849586 , -0.552282973545 ) circle ( 0.00263969600201 );
\draw[fill=white] ( 0.976318137491 , -0.551480811344 ) circle ( 0.0019797720015 );
\draw[fill=white] ( 0.979730103419 , -0.550879189694 ) circle ( 0.00148482900113 );

\draw[fill=white] ( 0.338490330802 , -0.600804712134 ) circle ( 0.046875 );
\draw[fill=white] ( 0.261406170504 , -0.628861052016 ) circle ( 0.03515625 );
\draw[fill=white] ( 0.2081253107 , -0.598099333266 ) circle ( 0.0263671875 );
\draw[fill=white] ( 0.178465433203 , -0.562752067702 ) circle ( 0.019775390625 );
\draw[fill=white] ( 0.161161966406 , -0.532781584063 ) circle ( 0.0148315429688 );
\draw[fill=white] ( 0.148184366308 , -0.510303721334 ) circle ( 0.0111236572266 );
\draw[fill=white] ( 0.133272238648 , -0.497790960514 ) circle ( 0.00834274291992 );
\draw[fill=white] ( 0.118672438539 , -0.497790960514 ) circle ( 0.00625705718994 );
\draw[fill=white] ( 0.108382945217 , -0.501536029809 ) circle ( 0.00469279289246 );
\draw[fill=white] ( 0.100665825227 , -0.50434483178 ) circle ( 0.00351959466934 );
\draw[fill=white] ( 0.0953317230359 , -0.507424477116 ) circle ( 0.00263969600201 );
\draw[fill=white] ( 0.0913311463929 , -0.509734211117 ) circle ( 0.0019797720015 );
\draw[fill=white] ( 0.0879191804644 , -0.510335832768 ) circle ( 0.00148482900113 );

\draw[fill=white] ( 0.529845780537 , -0.711283832549 ) circle ( 0.046875 );
\draw[fill=white] ( 0.488830155537 , -0.782324978953 ) circle ( 0.03515625 );
\draw[fill=white] ( 0.519591874287 , -0.835605838756 ) circle ( 0.0263671875 );
\draw[fill=white] ( 0.496520585224 , -0.875566483609 ) circle ( 0.019775390625 );
\draw[fill=white] ( 0.474275677101 , -0.902076932782 ) circle ( 0.0148315429688 );
\draw[fill=white] ( 0.46539847581 , -0.926466842876 ) circle ( 0.0111236572266 );
\draw[fill=white] ( 0.455665275737 , -0.943325239923 ) circle ( 0.00834274291992 );
\draw[fill=white] ( 0.453130047053 , -0.957703236264 ) circle ( 0.00625705718994 );
\draw[fill=white] ( 0.453130047053 , -0.968653086346 ) circle ( 0.00469279289246 );
\draw[fill=white] ( 0.454556113188 , -0.976740709288 ) circle ( 0.00351959466934 );
\draw[fill=white] ( 0.454556113188 , -0.982899999959 ) circle ( 0.00263969600201 );
\draw[fill=white] ( 0.453753950987 , -0.987449287864 ) circle ( 0.0019797720015 );
\draw[fill=white] ( 0.452568987656 , -0.99070494786 ) circle ( 0.00148482900113 );

\end{tikzpicture}
\endpgfgraphicnamed
 \\
$T$.
&
$X_2$.
\end{tabular}
\end{center}

Let $X_1=T$, $R_{(i)}=T_i$ and $S_1=S$. Let $h_{(0)}$ be a homeomorphism of $[-1,+1]^2$ with $T_0$ carrying top--right corner to top--right corner etc, and midpoints to midpoints. Let $h_{(i)}= \theta^i \circ h_{(0)}$ for $i=1,2,3$.

Inductively, suppose we have continuum $X_n$, geometric boundary $S_n$, and for each $\sigma \in \Sigma_n =\{0,1,2,3\}^n$ a rectangle $R_\sigma$ and a homeomorphism $h_\sigma$ of $[-1,+1]^2$ with $R_\sigma$. Fix a $\sigma$ for a moment. Then $R_\sigma$ has four subrectangles $h_\sigma(T_i)$. For $i=0,1,2,3$ let $h_{\sigma \smallfrown i}$ be a homeomorphism of $[-1,+1]^2$ with $h_\sigma (T_i)$ taking corners to corners etcetera. Let $R_{\sigma \smallfrown i} = h_{\sigma} (T_i) = R_\sigma \setminus h_{\sigma \smallfrown i} (U)$. Let $X_{n+1}= \bigcup_{\sigma \in \Sigma_n} \bigcup_{i=0}^3 R_{\sigma \smallfrown i} = \bigcup_{\sigma \in \Sigma_{n+1}} R_\sigma$. Let $S_{n+1}$ be the natural geometric boundary. Note that $X_{n+1} \subseteq X_n$.

Let $X=\bigcap_n X_n$. Then $X$ is a variant of Charatonik's description of Urysohn's locally connected, rational continuum in which every point has countably infinite order, see \cite{Chara}. Thus $X$ is rational (and locally connected). Since it is planar it is not $4$-sac. Note that each $R_\sigma \cap X$  has a countable boundary contained in the sides of $R_\sigma$. Call a side of $R_\sigma$ \emph{finite} if it contains only finitely many boundary points. A side containing infinitely many boundary points, is said to be \emph{infinite}.

For each $n$ and $\sigma$ in $\Sigma_n$, there are two circles, $h_\sigma(S^-)$ and $h_\sigma (S^+)$. Identify, for all $m \in \mathbb{Z}\cup\{\pm \infty\}$ and $i \in \{0,1\}$, the points $h_\sigma (p_{-m,i}^-)$ and $h_\sigma (p_{m,i}^+)$ (creating a `rational bridge' between the circles).  Note that the diameters of the circles shrink to zero with $n$. It follows that the resulting quotient space, $Y$, is a locally connected, rational continuum. We show that $Y$ is $\omega$-sac.

Fix distinct points $x_1, \ldots , x_n$ in $Y$. The diameters of the rectangles, $R_\sigma$ for $\sigma \in \Sigma_m$, shrink to zero with $m$, so we can find an $N$ such that if $i \ne j$, $x_i \in R_\sigma$, and $x_j \in R_\tau$ where $\sigma, \tau \in \Sigma_N$ then $R_\sigma$ and $R_\tau$ are disjoint. For each $i$, let $R_i$ be the unique $R_\sigma$ containing $x_i$.

Subdivide the square $r=(-1,+1)^2$ into four subsquares $r_{(i)}=\theta^i \left([0,1]^2 \right) \cap r$. And continue subdividing to get a final subdivision of $(-1,+1)^2$ into subsquares $r_\sigma$ for $\sigma \in \Sigma_N$. Note that two squares $r_\sigma$ and $r_\tau$ are adjacent if and only if the corresponding rectangles $R_\sigma$ and $R_\tau$ are adjacent. For each $\sigma$ in $\Sigma_N$, consider $R_\sigma$. It has four sides, at most two are finite sides. For each finite side remove the line segment in $r$ which is the corresponding side in $r_\sigma$. The result $r'$ is an open, connected subset  of the plane. It follows that $r'$ is $\omega$-sac. Hence there is an arc $\alpha'$ which visits the interior of the squares $r_i$ in order: $r_1, r_2, \ldots , r_n$ (indeed we can suppose $\alpha'$ visits the centers of the $r_i$ in turn). Further, we can suppose that $\alpha'$ consists of a finite union of horizontal or vertical line segments of the form $\{p/q\} \times J$ or $J \times \{p/q\}$ where $p \in \mathbb{Z}$, $q \in \mathbb{N}$ and $J$ is a closed interval. Let $M$ be a common denominator of all the denominators ($q$s) used. Then $\alpha'$ is an arc on the grid $r' \cap \left( \left(\bigcup_{p \in \mathbb{Z}} \{p/M\} \times \mathbb{R}\right) \cup \left(\bigcup_{p \in \mathbb{Z}} \mathbb{R} \times \{p/M\} \right) \right)$.

Consider $X_1$. There is a connected chain of circles, $V_0$, in $X_1$ from the bottom edge to the top, and a connected chain of circles, $H_0$, from the left side to the right. Note that $V_0$ and $H_0$ are in $X$. Now consider $X_2$. There is a connected chain of circles to the right of $V_0$ from the top edge to a circle in $H_0$, and another to the right of $V_0$ from the bottom edge to a circle in $H_0$. By construction, both chains end at the {\sl same circle} of $H_0$. Call the union of these two chains, along with the circle they connect to, $V_1$. It is a vertical connected chain of circles from the top edge to the bottom. Similarly, there is a vertical connected chain of circles, $V_{-1}$ from the top edge to the bottom, lying to the left of $V_0$. Further there are two horizontal connected chains of circles, $H_{1}$ and $H_{-1}$, above and, respectively, below, $H_0$. Observe that $V_0$ and $V_{\pm 1}$ are disjoint, as are $H_0$ and $H_{\pm 1}$. Together these six chains form a three--by--three `grid' in $X$.
Repeating, we find a `grid' of horizontal, $H_{\pm n}$ and vertical, $V_{\pm n}$, chains, all in $X$, where the $H_{\pm n}$ converges to the left and right  sides, $\{\pm 1\} \times [-1,1]$ and $V_{\pm n}$ converges to the top and bottom edges $[-1,1] \times \{\pm 1\}$.    

Now consider a rectangle $R_\sigma$ for some $\sigma$ in $\Sigma_N$. It has at least two `infinite' sides. For concreteness let us suppose that the bottom and right sides of $R_\sigma$ are infinite, with the limit point on the bottom edge being to the right, and the limit along the right edge being at the top (all other cases are very similar). The vertical chains, $V_n$ in $X$, for $n \in \mathbb{N}$, have analogues in $R_\sigma$. By construction, each $V_n$ meets the bottom edge of $R_\sigma$ in an arc of a circle whose ends are points in the $R_\tau$ `below' $R_\sigma$. Extend $V_n$ to include this arc. Repeat at the top edge, if it is infinite. Apply the same procedure to the horizontal chains, $H_n$ for $n \in \mathbb{N}$. The horizontal chains, and respectively the vertical chains, remain disjoint. Note that, by construction, if $R_\tau$ is the rectangle `below' $R_\sigma$, then the $n$th vertical chain in $R_\sigma$ connects to the $n$th vertical chain in $R_\tau$ (and similarly for the rectangle to the right of $R_\sigma$). If $x_i$ is in  $R_\sigma$ but not on the geometric boundary of $R_\sigma$, then let $P_i$ be sufficiently large that $x_i$ is to the left of $V_{P_i}$ and below $H_{P_i}$. 

Now let $P$ be the maximum of the $P_i$. Return to an individual rectangle, $R_\sigma$, as in the previous paragraph. Take the union of the vertical chains, $V_P, \ldots , V_{P+M}$, and the horizontal chains, $H_P, \ldots , H_{P+M}$. Take the union now over all $\sigma$ in $\Sigma_N$. This gives a `grid', $G$,  naturally containing an isomorphic copy of the grid $G'$ in $r'$. Think of the grid, $G'$, as a graph, and $\alpha'$ as an edge arc in this graph. Then we can realize the arc $\alpha'$ in $G'$ as a connected chain of circles in $X$. Evidently (by traveling along the `top' or `bottom' edges of the circles in the union) we can extract an arc, $\alpha_\ast$, contained in this union. The arc $\alpha_\ast$ visits the $R_i$ in order. Note that $\alpha_\ast$ is not (necessarily) an arc in $Y$, but it can easily be modified to be so, call this arc, $\alpha_0$.

To complete the proof,   we modify $\alpha_0$, to another arc $\alpha$ in $Y$,  which visits the points $x_i$ in order. As $\alpha_0$ visits the $R_i$ in turn, there are subarcs $\beta_i$ of $\alpha_0$, where $\beta_i$ comes before $\beta_j$ if $i < j$, such that $\beta_i$ crosses from one infinite edge of $R_i$ to another (along the `grid' $G$ inside $R_i$). We will replace $\beta_i$ in $\alpha_0$ by another subarc, visiting $x_i$, contained inside $R_i$, with the same start and end as $\beta_i$, but otherwise disjoint from the `grid' $G$. Doing this for all $i$, gives the arc $\alpha$ in $Y$.

Fix $i$. Again for concreteness, orient $R_i=R_\sigma$ as above. Suppose that $\beta_i$ enters $R_i$ at $y_i$, a point on the bottom edge, and exits at $z_i$, a point on the right edge. Pick $Q$ in $\mathbb{N}$ sufficiently large that, $V_Q$ is to the right of the rightmost vertical chain in $G \cap R_i$, above the highest horizontal chain in $G \cap R_i$ (i.e. $Q > P+M$) and if $x_i$ is not on the geometric boundary of $R_i$, the  vertical chain $V_{-Q}$ is to the left of $x_i$ and the horizontal chain $H_{-Q}$ is below $x_i$. The union of $V_{\pm Q}$ and $H_{\pm Q}$ contains an obvious `ring', a connected cycle of circles, just interior to the geometric boundary of $R_i$. Observe that this ring meets each arc component of $\alpha_0 \cap R_i$ in two circles, which are bridges. Select a simple closed curve (in $Y$), $S$, contained in this ring, which connects with $\beta_i$ at two points (one, call it $y_i'$, near $y_i$, and another, call it $z_i'$, near $z_i$), but which uses the bridges to prevent intersection with any other (arc component of) $\alpha_0$. We will modify $S$ so that it visits $x_i$. If this is possible then either the arc `travel along $\beta_i$ from $y_i$ to $y_i'$, then {\sl clockwise} along $S$ until we reach $z_i'$, followed by traveling along the arc $\beta_i$ to $z_i$'; or the arc obtained by following $S$ {\sl anti--clockwise}, is the required modification of $\beta_i$.

Two cases arise. If $x_i$ is on the geometric boundary of $R_i$, we can find two arcs starting at $x_i$, and otherwise disjoint, both meeting $S$ (but disjoint from the grid $G$). The required modification of $S$ is now obvious (follow $S$, then the first arc met, to $x_i$, back to $S$ along the second arc, and finish following $S$). If $x_i$ is not on the geometric boundary, then it is to the left and below the grid. It is also in some rectangle, $R_\tau$, $\tau$ from some $\Sigma_{N'}$ where $N' >N$, where $R_\tau$ is disjoint from the geometric boundary of $R_i$. Following $S$ anticlockwise we can get to a point, $a_i$, below the lowest horizontal line of the grid, but above and to the left of the top--left corner of $R_\tau$. Following $S$ clockwise we can get to a point, $b_i$, left of  the leftmost vertical  line of the grid $G \cap R_i$, but below and to the right of the bottom--right corner of $R_\tau$. We can now find disjoint arcs from $a_i$ to the top--left corner of $R_\tau$, and from $b_i$ to the bottom--right corner of $R_\tau$. And these can be extended to disjoint (except at $x_i$) arcs $a_i$ to $x_i$ and $b_i$ to $x_i$. Again, using these arcs, we can modify $S$ to detour through $x_i$.
\end{proof}

\section{Complexity}
As discussed at the end of Section~2, the definition of the $n$-sac property uses quantification over two uncountable sets, namely all $n$-tuples of points in the space, and all arcs in the space. However there is a characterization of the $3$-sac property in graphs which only requires quantification  over countable sets. This characterization, then, is hugely simpler than the formal definition. One might hope to find similarly simple characterizations of $n$-sac, or $\omega$-sac, regular curves or rational curves. 
In this section we show that no such simple characterization of the $n$-sac (or $\omega$-sac) property exists for rational curves. Further,   there is no characterization of $n$-sac or $\omega$-sac general curves which does not use quantification over two uncountable sets, in other words, no  characterization simpler than the definition.

More precisely we show that for $n\ge 2$ or $n=\omega$ the set of $n$-sac rational continua and the set of all $n$-sac but not $(n+1)$-sac rational continua are $\mathbf{\Sigma}_1^1$-hard subsets of the space of  subcontinua of $\mathbb{R}^N$, for $N\ge 3$. This should be interpreted as saying that there is no formula characterizing the $n$-sac, or $\omega$-sac, rational curves which does not require at least one (existential) quantifier running over an uncountable (Polish) space.
We further show that, for $n\ge 2$ or $n=\omega$, the set of $n$-sac  continua are $\mathbf{\Pi}_2^1$-complete subsets of the space of  subcontinua of $\mathbb{R}^N$, for $N\ge 4$. The logical interpretation of this statement is that the simplest formula characterizing the $n$-sac, or $\omega$-sac, curves has exactly one universal quantifier followed by one existential quantifier running over  uncountable Polish spaces --- as in the formal definition of the $n$-sac property.

These complexity results are part of a significant body of work in descriptive set theory. In particular, Becker proved that the set of simply connected continua in $\R^3$ form a $\mathbf{\Pi}_1^1$-complete set (see \cite{kech} 33.17), while Becker and Ajtai (independently) showed that the path connected (i.e. arc connected, or equivalently $2$-sac) continua in $\R^3$ are $\mathbf{\Pi}_2^1$-complete  (see \cite{kech} 37.11). Other recent work in this area includes that of Becker and Pol \cite{becpol} on arc components.

The basic notions on this section are taken from \cite{cam}. A Polish space is one which is separable and completely metrizable. For a Polish  space $X$,  denote by $\mathcal{C}(X)$ the {\it hyperspace of subcontinua of} $X$
endowed with the Vietoris topology. It is Polish. A subset $A$ of $X$ is called {\it analytic} if it is the continuous image of a Polish space. The symbol $\sigmaset{1}{1}(X)$ denotes the
family of analytic subsets of $X$. The set of complements of elements of $\mathbf{\Sigma}_1^1$ is denoted $\mathbf{\Pi}_1^1$, while the complements of continuous images of $\mathbf{\Sigma}_1^1$ is written $\mathbf{\Pi}_1^2$.

Let $X$ and $Y$ be two metric spaces, and $A\subseteq X$, $B\subseteq Y$, we say that $A$ is {\it Wadge reducible to} $B$ if there is a continuous map
$f:X\to Y$ such that $A=f^{-1}(B)$; and we denote this by $A\leq_WB$.
Let $Y$ be a Polish space. Let $\Gamma$ be any of the families of subsets mentioned above. Then a subset $B$ of $Y$
is $\Gamma$--{\it hard} if $A\leq_WB$ for any $A\in \Gamma (X)$, where $X$ is a zero-dimensional Polish space. If in addition
$B\in \Gamma (Y)$, then we say that $B$ is $\Gamma$-{\it complete}. The point here is that, for example, any $\sigmaset{1}{1}$-hard set is at least as complex as any analytic set, and any formula describing it must contain at least one existential quantifier running over an uncountable Polish space. Further, if a $\Gamma$-hard set, $A$ say, Wadge reduces to another set $B$, then $B$ is also $\Gamma$-hard. This gives a standard method of proving that a given set of interest is, say, a $\sigmaset{1}{1}$-hard set --- show that a known $\sigmaset{1}{1}$-hard set reduces onto it.

Let $\N^{<\N}$ be the set of all finite sequences on $\N$, including the empty sequence, $()$.  Given $s=(s_1, s_2, \ldots, s_n)$ and $k$, let $s\concat k = (s_1, \ldots , s_n,k)$.
A {\it tree on} $\N$ is a subset $\tau$ of  $\N^{<\N}$ which is closed under initial segments, in other words if $t\in \tau$ and $s=t\restrict m$ for some $m\leq length(t)$, then $s\in \tau$. Identifying a subset of $\N^{<\N}$ with its characteristic function, let $\mathbf{Tr}$ be the subspace of $\{0,1\}^{\N^{<\N}}$ of all trees on $\N$. It is a closed subset, hence compact.
A tree with an infinite branch  is said to be \emph{ill-founded}. 
Denote by  $\mathbf{IF}$  the space of all ill-founded trees on $\N$. It is well known  that $\mathbf{IF}$ is $\sigmaset{1}{1}$-complete, see Theorem 27.1 and page 240 of \cite{kech}. 

Now our approach to showing that, for $n\ge 2$ or $n=\omega$ the set of $n$-sac rational continua and the set of all $n$-sac but not $(n+1)$-sac rational continua are $\mathbf{\Sigma}_1^1$-hard sets, is clear --- we will find  continuous reductions from $\mathbf{IF}$ onto these sets. This entails constructing, for a given tree, a suitable rational continua. The next few lemmas provide building blocks (`tiles') and tools for making complex rational continua.

A {\it tile} is any space $T$ which is (i) a subspace of the solid square pyramid in $\R^3$ with base $S=[-1,+1]^2 \times \{0\}$ and vertex at $(0,0,1)$ (so it has height $1$) and (ii) contains the four corner points of the base, $(i,j)$ for $i,j=\pm 1$. 
Call the intersection of a tile $T$ with $S$, the {\it base of} $T$. Call the intersection of $T$ with the {\it boundary} $B=\left([-1,1]\times \{-1,1\}\times \{0\}\right) \cup \left(\{-1,1\} \times [-1,1] \times \{0\}\right)$ of the base $S$, the {\it boundary} of $T$. Call the point $(-1,1, 0)$ the {\it top--left corner of the base}.

\begin{lemma}
There are (homeomorphic) tiles $T_0$ and $T_1$  such that: (i) $T_0$ and $T_1$ are $\omega$-sac rational curves,  (ii) the boundary of $T_0$ is  $B$, and (iii) the boundary of $T_1$  is $\left(A\times \{-1,1\}\times \{0\}\right) \cup \left(\{-1,1\} \times A \times \{0\}\right)$ where $A=\{-1,0,1\} \cup \{ -2^{-n} : n \in \N\}$.  
\end{lemma}

\begin{proof}
The example, $Y$, of an $\omega$-sac rational curve given in Theorem~\ref{omega_rational} is derived from a space $X$. This space $X$ is a subspace of $[-1,+1]^2$. We may suppose that $X$ is in fact a subspace of the square $S=[-1,+1]^2 \times \{0\}$. The space $Y$ is obtained from $X$ by identifying a sequence of pairs of double sequences. These double sequences all are disjoint from the boundary, $B$, of the square $S$, and the diameters and distance between pairs of sequences converges to zero. This identification process can be repeated in $(-1,+1)^2 \times \R$, keeping the boundary, $B$, of the square, $S$, fixed, to get a space $T_0'$ homeomorphic to $Y$. 
Applying a homeomorphism of $[-1,+1]^2 \times \R$ fixing $B$, and changing only the $z$--coordinates, to $T_0'$, we get a space $T_0$, also homeomorphic to $Y$ and containing $B$, and which is contained in the pyramid with base $[-1,+1]^2 \times \{0\}$ and  height $1$. Thus $T_0$ is a tile.

Scaling $\R^3$ around the center point of the the base square, $S$, we can shrink $T_0$ away from the boundary $B$ of $S$ and still have it inside the required pyramid. Instead of doing this transformation, shrink $T_0$ while keeping fixed the set  $\left(A\times \{-1,1\}\times \{0\}\right) \cup \left(\{-1,1\} \times A\times \{0\}\right)$. This gives $T_1$.
\end{proof}

Let $X$ be a space and $A$ an infinite subset. We say that  $X$ is $\omega$-sac$^+$ (with respect to $A$) if for any points $x_1, \ldots , x_n$ in $X$ there is an arc $\alpha$ in $X$ visiting the $x_i$ in order, such that $\alpha$ meets $A$ only in a finite set.  Observe that if $X$ is $\omega$-sac$^+$ with respect to $A$, and $A'$ is an infinite subset of $A$, then $X$ is $\omega$-sac$^+$ with respect to $A'$.

\begin{lemma}[$\omega$--Gluing]\label{omega_gluing}
Let $Z = X \cup Y$, where $X, Y$ and $A= X \cap Y$ are infinite. If $X$ is $\omega$-sac$^+$ with respect to $A$, and $Y$ is $\omega$-sac, then $Z$ is $\omega$-sac.
\end{lemma}

\begin{proof}
Take any finite sequence of points $z_1, \ldots , z_{N}$ in $Z$. By adding points to the start and end of the sequence, if necessary, we can suppose that $z_0$ and $z_N$ are in $X$. Group the sequence, $z_1, \ldots , z_{n_1}$, $z_{n_1+1}, \ldots , z_{n_2}$, $\ldots, z_{n_{k-1}}, z_{n_{k-1}+1}$, $\ldots, z_{n_k}$, where $z_1, \ldots , z_{n_1}$ are in $X$, $z_{n_1+1}, \ldots , z_{n_2}$ are in $Y \setminus X$, and so on, until $z_{n_{k-1}+1}, \ldots, z_{n_k}=z_N$ are in $X$. Pick $t_1^{\pm}, \ldots, t_k^{\pm}$ in $A \setminus \{z_i\}_{i \le N}$.

Using the fact that $X$ is $\omega$-sac$^+$, pick arc $\alpha^-$ in $X$ visiting in order, $z_1, \ldots, z_{n_1}$, $t_1^-, t_1^+$, $z_{n_2+1}, \ldots, z_{n_2+1}$, $\ldots, z_{n_3}, t_2^-, t_2^+$ and so on, ending with $z_{n_k}$, such that $\alpha^-$ meets $A$ only in a finite set $F$.

Using the fact that $Y$ is $\omega$-sac, pick an arc $\alpha^+$ in $Y$ visiting in order the points, $t_1^-, z_{n_1+1}$, $\ldots , z_{n_2}, t_1^+,t_2^-$ and so on,  avoiding $F \setminus \{t_1^{\pm}, \ldots, t_k^{\pm}\}$.

Now we can interleave $\alpha^-$ and $\alpha^+$ to get an arc, $\alpha$, visiting all the specified points in order. So we start $\alpha$ by following $\alpha^-$ to visit $z_1, \ldots , t_1^-$, then pick up $\alpha^+$ at $t_1^-$ to visit $z_{n_1+1}, \ldots , z_{n_2}, t_1^+$, and back to $\alpha^-$ from $t_1^+$, and so on.
\end{proof}

\begin{lemma}\label{tiles_om_sac_plus} \

(i) The tile $T_0$ is  $\omega$-sac$^+$ with respect to any infinite discrete subset of its boundary.

(ii) The tile $T_1$ is $\omega$-sac$^+$ with respect to its boundary.
\end{lemma}
\begin{proof}
Recall that $T_0$ and $T_1$ are both homeomorphic.  In turn, $T_0$ is a homeomorph of $Y$ from Theorem~\ref{omega_rational} with the boundary square for both not just homeomorphic but identical (when we identify the plane, $\R^2$, with $\R^2 \times \{0\}$). So we argue this for $Y$ only. Looking at the proof that $Y$ is $\omega$-sac it is clear that the arc, $\alpha_0$, visiting some specified points, $x_1, \ldots, x_n$, in order, need only touch the boundary in an arbitrarily small neighborhood of any $x_i$ which happens to be on the boundary. This immediately gives the first claim --- $Y$ (and so $T_0$) is $\omega$-sac$^+$ with respect to infinite discrete subsets of the boundary square.

Further, the point $(0,-1)$  can be reached from the interior of $Y$ (away from the boundary square) by two disjoint arcs which meet the set $\left(A\times \{-1,1\}\right) \cup \left(\{-1,1\} \times A\right)$  only at $(0,-1)$ --- for one arc, $\alpha^-$, follow one side of the sequence of circles converging to $(0,-1)$ and for the other, $\alpha^+$, start at $(0,-1)$ go right along the boundary edge a short way, and then go into the interior. The same is true for the points $(0,1), (-1,0)$, and $(1,0)$.

Now to get the desired arc, if every $x_i$ is not one of $(0,-1), (0,1), (-1,0)$, or $(1,0)$, then just use $\alpha_0$. While if $x_i$, is say, $(0,-1)$, then pick $\alpha_0$ to visit $x_1, \ldots, x_{i-1}, t^-, t^+, x_{i+1}, \ldots$, where $t^-, t^+$ are points close to $(0,-1)$ on $\alpha^-$ and $\alpha^+$ respectively. Now let $\alpha$ be the arc that follows $\alpha_0$ to $t^-$, then follows $\alpha^-$ to $x_i=(0,-1)$, then $\alpha^+$ to $t^+$, and then resumes along $\alpha_0$. 
\end{proof}

For any tile $T$, $\vx=(x,y)$ in $\R^2$ and $a,b >0$, denote by $T(\vx, a,b)$ the space $T$ scaled in the $x$ and $y$ coordinates so its base has length $a$ and width $b$, then scaled in the $z$ coordinate so that the pyramid containing it has height no more than the smaller of $a$ and $b$, and then translated in the $x,y$-plane so that the top--left corner is at $(x,y,0)$.

From Lemma~\ref{omega_gluing}, part (ii) of Lemma~\ref{tiles_om_sac_plus}, and an easy induction argument, the following is clear.
\begin{lemma}\label{edge_conn_T1}
Any space obtained by gluing along matching edges a finite family of translated and scaled copies of $T_1$ is a rational $\omega$-sac curve.
\end{lemma}

We define recursively a sequence of tiles. The first in the sequence is $T_1$ from above. Given tile $T_n$, where $n \ge 1$, define $T_{n+1}$ to be $T_n ((-1,1),1,1) \cup T_n ((-1,0),1,1) \cup T_n ((0,1),1,1) \cup T_n ((0,0),1,1)$ scaled in the $z$--coordinate only so as to fit inside the pyramid with base $S$ and height $1$. Then all the tiles $T_n$ are rational $\omega$-sac continua.

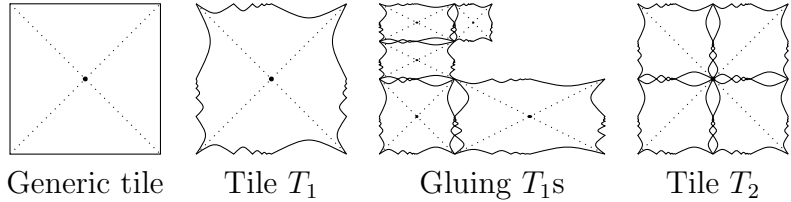
\newcommand{\gtile}[5]{
\begin{scope}[cm={#3/2,0,0,#4/2,(#1+#3/2,#2-#4/2)}]
\tile{#5}
\end{scope}
}
\newcommand{\gtileP}[5]{
\begin{scope}[cm={#3/2,0,0,#4/2,(#1+#3/2,#2-#4/2)}]
\tileP{#5}
\end{scope}
}
\newcommand{\pyramid}{
\draw [dotted] (-1,-1) -- (1,1);
\draw [dotted] (-1,1) -- (1,-1);
\draw [fill]   (0,0) circle (0.025);
}
\newcommand{\tile}[1]{

\ifthenelse{\equal{#1}{0}}
{
\draw (-1,-1) rectangle (1,1);
}{}

\ifthenelse{\equal{#1}{1}}
{
\draw (-1,-1) 
.. controls (-3/4,-5/6) .. (-1/2,-1)
.. controls ( -3/8,-7/8) .. (-1/4,-1)
.. controls ( -3/16,-15/16) .. (-1/8,-1)
.. controls ( -3/32,-31/32) .. (-1/16,-1)
.. controls ( -1/32,-63/64) .. (0,-1)
.. controls (1/2,-3/4) ..  (1,-1);
\draw (-1,1) 
.. controls (-3/4,5/6) .. (-1/2,1)
.. controls ( -3/8,7/8) .. (-1/4,1)
.. controls ( -3/16,15/16) .. (-1/8,1)
.. controls ( -3/32,31/32) .. (-1/16,1)
.. controls ( -1/32,63/64) .. (0,1)
.. controls (1/2,3/4) ..  (1,1);
\draw (-1,-1) 
.. controls (-5/6,-3/4) .. (-1,-1/2)
.. controls (-7/8,-3/8) .. (-1,-1/4)
.. controls (-15/16,-3/16) .. (-1,-1/8)
.. controls (-31/32,-3/32) .. (-1,-1/16)
.. controls (-63/64,-1/32) .. (-1,0)
.. controls (-3/4,1/2) ..  (-1,1);
\draw (1,-1) 
.. controls (5/6,-3/4) .. (1,-1/2)
.. controls (7/8,-3/8) .. (1,-1/4)
.. controls (15/16,-3/16) .. (1,-1/8)
.. controls (31/32,-3/32) .. (1,-1/16)
.. controls (63/64,-1/32) .. (1,0)
.. controls (3/4,1/2) ..  (1,1);
}{}
\ifthenelse{\equal{#1}{2}}
{
\foreach \x in {-1,0}
  \foreach \y in {0,1}
      {\gtile{\x}{\y}{1}{1}{1};}
}{}
\ifthenelse{\equal{#1}{3}}
{
\foreach \x in {-1,-1/2,0,1/2}
  \foreach \y in {-1/2,0,1/2,1}
      {\draw[thin] (\x,\y) rectangle +(1/2,-1/2);}
}{}
\ifthenelse{\equal{#1}{4}}
{
\fill[pattern=checkerboard light gray] (-1,1) rectangle (1,-1);
}{}
}
\newcommand{\tileP}[1]{
\tile{#1}
\pyramid
}

\begin{center}
\begin{tabular}{cccc}
\begin{tikzpicture}
\tileP{0}
\end{tikzpicture}
& \begin{tikzpicture}
\tileP{1}
\end{tikzpicture}
& \begin{tikzpicture}
\gtileP{0}{1}{1}{1}{1}
\gtileP{1}{1}{2}{1}{1}
\gtileP{0}{1.5}{1}{1/2}{1}
\gtileP{0}{2}{1}{1/2}{1}
\gtileP{1}{2}{1/2}{1/2}{1}
\end{tikzpicture}
& \begin{tikzpicture}
\tileP{2}
\end{tikzpicture} \\
Generic tile 
& Tile $T_1$ 
& Gluing $T_1$s
& Tile $T_2$
\end{tabular}
\end{center}

\begin{theorem}\label{rationalsac}
Fix $N \ge 3$. For $n \ge 2$ or $n=\omega$, let $R_n$ be the set of rational $n$-sac continua, and let $R_{n, \neg (n+1)}$ be the set of rational continua which are $n$-sac but not $(n+1)$-sac.

Then all the sets $R_n$ and $R_{n, \neg (n+1)}$ are $\sigmaset{1}{1}$--hard subsets of the space $\mathcal{C}(\R^N)$.
\end{theorem}

\begin{proof}
 We prove that there is a continuous map $K$ of the space $\mathbf{Tr}$ of all trees on $\N$ into the space $\mathcal{C}(\R^3)$ such that: if the tree $\tau$ has no infinite branch then $K_\tau$ is a rational continuum which is not arc-connected (in other words, $2$-sac), while if $\tau$ has an infinite branch, then $K_\tau$ is an $\omega$-sac rational continuum. The claim that $R_n$ is a $\sigmaset{1}{1}$--hard subset of the space $\mathcal{C}(\R^N)$, follows simultaneously for all $n$ and $N$.
We then give the minor modifications necessary to have that $K_\tau$ is $n$-sac but not $(n+1)$-sac when $\tau$ has an infinite branch. The remaining claims follows immediately.

A basic building block for  $K_\tau$ is $S(T)$ a variant of the  topologist's sine-curve based on a tile $T$. This sine-curve lies in the rectangular box 
$\{(x,y,z)\in\R^3 : 0\leq x\leq 13/3, 0\leq y \leq 5/2, 0\leq z\leq 1\}$. We call the point $(0,5/2,0)$ the top left corner of $S(T)$.

Explicitly $S(T)$ is  $(D_0\cup AB_0 \cup U_0 \cup AT_0) \cup \bigcup_{n \ge 1} (D_n \cup C_n \cup AB_n \cup U_n \cup AT_n)$ where
\begin{eqnarray*}
D_n &=& \bigcup_{i=1}^{2 \cdot 4^n} T(10/(3 \cdot 4^n), 1/2+i/4^n,1/4^n),\\
C_n &=& T(10/(3 \cdot 4^n), 1/2-1/4^n,1/4^n) \cup T(10/(3 \cdot 4^n), 1/2,1/4^n) \quad (n \ge 1), \\
AB_n&=& T((7/(3 \cdot 4^n),1/2+1/4^n),1/4^n),  \\
U_n &=& \bigcup_{i=1}^{2 \cdot 4^n -1} T((4/(3 \cdot 4^n),1/2+i/4^n),1/4^n) \\
    & & \cup \, T((4/(3 \cdot 4^n),5/2-1/4^{n+1},1/4^n,3/4^{n+1}) \\
    & & \cup \, T((4/(3 \cdot 4^n),5/2,1/4^n,1/4^{n+1}), \quad \text{and} \\
AT_n&=& T((4/(3 \cdot 4^n)-1/4^{n+1},5/2,1/4^{n+1},1/4^{n+1}).
\end{eqnarray*}

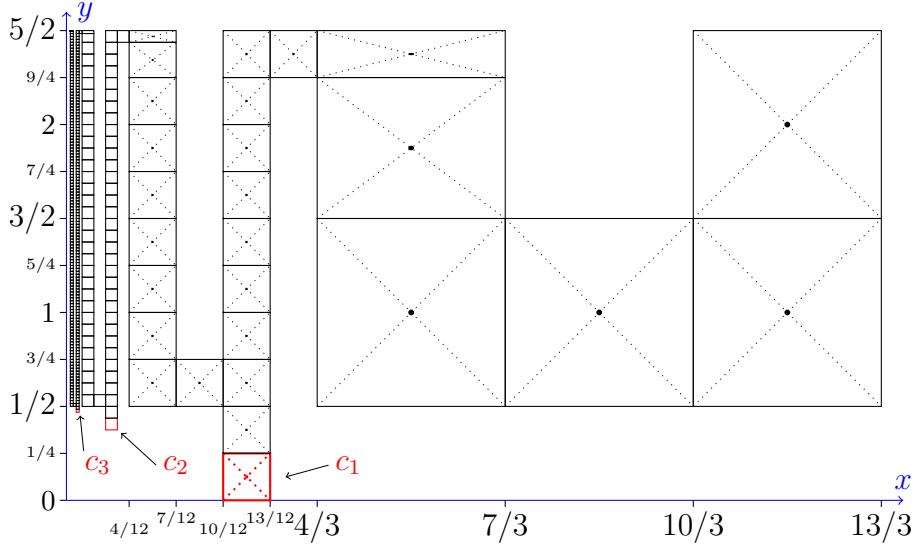
\tikzset{c1/.style={}}
\tikzset{c2/.style={}}
\tikzset{c3/.style={}}
\newcommand{\SineCurve}[1]{
\foreach \i in {1,...,2}
  {\gtileP{10/3}{1/2+\i}{1}{1}{#1}}
\gtileP{7/3}{1/2+1}{1}{1}{#1}
\foreach \i in {1,...,1}
  {\gtileP{4/3}{1/2+\i}{1}{1}{#1}}
\gtileP{4/3}{5/2-1/4}{1}{3/4}{#1}
\gtileP{4/3}{5/2}{1}{1/4}{#1}
\gtileP{13/12}{5/2}{1/4}{1/4}{#1}

\begin{scope}[c1]
{\gtileP{10/12}{1/2-1/4}{1/4}{1/4}{#1}}
\end{scope}
\foreach \i in {0,...,8}
  {\gtileP{10/12}{1/2+\i/4}{1/4}{1/4}{#1}}
\gtileP{7/12}{1/2+1/4}{1/4}{1/4}{#1}
\foreach \i  in {1,...,7}
  {\gtileP{4/12}{1/2+\i/4}{1/4}{1/4}{#1}}
\gtileP{4/12}{5/2-1/16}{1/4}{3/16}{#1}
\gtileP{4/12}{5/2}{1/4}{1/16}{#1}
\gtile{13/48}{5/2}{1/16}{1/16}{#1}

\begin{scope}[c2]
\gtile{10/48}{1/2-1/16}{1/16}{1/16}{0}
\end{scope}
\foreach \i in {0,...,32}
  {\gtile{10/48}{1/2+\i/16}{1/16}{1/16}{0}}
\gtile{7/48}{1/2+1/16}{1/16}{1/16}{0}
\foreach \i in {1,...,31}
  {\gtile{4/48}{1/2+\i/16}{1/16}{1/16}{0}}
\gtile{4/48}{5/2-1/64}{1/16}{3/64}{0}
\gtile{4/48}{5/2}{1/16}{1/64}{0}
\gtile{13/192}{5/2}{1/64}{1/64}{0}

\begin{scope}[c3]
\gtile{10/192}{1/2-1/64}{1/64}{1/64}{0}
\end{scope}
\foreach \i in {0,...,128}
  {\gtile{10/192}{1/2+\i/64}{1/64}{1/64}{0}}
\gtile{7/192}{1/2+1/64}{1/64}{1/64}{0}
\foreach \i in {1,...,127}
  {\gtile{4/192}{1/2+\i/64}{1/64}{1/64}{0}}
\gtile{4/192}{5/2-1/256}{1/64}{3/256}{0}
\gtile{4/192}{5/2}{1/64}{1/256}{0}
\gtile{13/768}{5/2}{1/256}{1/256}{0}
}
\newcommand{\gSineCurve}[3]{
\begin{scope}[cm={#2,0,0,#2,(0,#1-#2/2-#2/2-#2/2-#2/2-#2/2)}]
\SineCurve{#3}
\end{scope}
}

\begin{center}
\beginpgfgraphicnamed{fig1_Sinecurve}

\begin{tikzpicture}[xscale=2.5, yscale=2.5]

\draw[->] [blue] (0,0) -- (4.45,0) node[above] {$x$};
\draw[->] [blue] (0,0) -- (0,2.6) node[right] {$y$};
\foreach \pos in {4/3,7/3,10/3,13/3}
	\draw[shift={(\pos,-1pt)}] (0pt,0pt) -- (0pt, 1pt) node[below] {$\pos$};
\foreach \pos in {7/12,13/12}
	\draw[shift={(\pos,-1pt)}] (0pt,0pt) -- (0pt, 1pt) node[below] {\tiny{$\pos$}};
\foreach \pos in {4/12,10/12}
	\draw[shift={(\pos,-1pt)}] (0pt,0pt) -- (0pt, 1pt) node[below=4pt] {\tiny{$\pos$}};
\foreach \pos in {1/2,3/2,5/2}
         \draw[shift={(-1pt,\pos)}] (0pt,0pt) -- (1pt, 0pt) node[left] {$\pos$};
\foreach \pos in {0,1,2}
         \draw[shift={(-1pt,\pos)}] (0pt,0pt) -- (1pt, 0pt) node[left] {\small{$\pos$}};
\foreach \pos in {1/4,3/4,5/4,7/4,9/4}
         \draw[shift={(-1pt,\pos)}] (0pt,0pt) -- (1pt, 0pt) node[left] {\tiny{$\pos$}};

\begin{scope}[c1/.style={red,thick}, c2/.style={red}, c3/.style={red}]
\SineCurve{0}
\end{scope}
\node [circle,red,inner sep=1pt] (c1n) at (3/2,3/16) {$c_1$};
\draw [very thin,->] (c1n.west) -- (14/12,1/8);

\node [circle,red,inner sep=1pt] (c2n) at (7/12,3/16) {$c_2$};
\draw [very thin,->] (c2n.west) -- (14/48,1/2-1/7);

\node [circle,red,inner sep=1pt] (c3n) at (1/6,3/16) {$c_3$};
\draw [very thin,->] (c3n.north west) -- (14/192,1/2-1/20);

\end{tikzpicture}
\endpgfgraphicnamed

\end{center}
For each $m \ge 1$, let $c_m=T(10/(3 \cdot 4^n), 1/2-1/4^n,1/4^n)$ be the tile in $S(T)$ at the bottom of the $m$th connector, $C_m$.

For any point $y$ in $\R$, and any $a>0$, let $S(y,a,T)$ be the sine curve $S(T)$ scaled (in all directions) by $a$, and translated so that its top left corner is at $(0,y,0)$.

Next, given a tree $\tau$ and a tile $T$, we define a `branch space', $B(T,\tau)$, 
 lying in the rectangular box, $\{(x,y,z)\in\R^3 :0\leq x\leq 13/3, 0\leq y \leq 10/3, 0\leq z\leq 1\}$,
 which is $\bigcup \{ S_s : s \in \tau\}$, where each $S_s$ is defined with the aid of some connecting tiles, $c_s$, and numbers, $y_s$, by induction on the height of $s$, as follows.
\begin{itemize}
\item[Step 1] Let $y_{()}=5/2+5/6=10/3$, and let $S_{()}=S(y_{()},1,T)$ (i.e. the sine-curve defined above based on $T$, translated along the $y$-axis by $5/6$). 

\item[Step 2]  The sine curve, $S_{()}$ has a family of connecting tiles $c_m$. Set $c_{(m)}=c_m$. Let $y_{(m)}=y_{()} - 2/4^0 - 2/4^m$, and let 
$S_{(m)}=S(y_{(m)},\frac{1}{4^m},T)$. Note, critically, that the top--right tile of this sine-curve, $S_{(m)}$, is such that its top edge coincides with the bottom edge of $c_{(m)}$.

\item[Step $n+1$] Fix an $s\in\tau$ with length $n$. We again will have connecting tiles, $c_m$, from the sine-curve $S_{s}$. Set  $c_{s\concat m}=c_m$.
Let $y_{s\concat m}=y_{s} - 2/4^L - 2/4^{L+m}$ where $L=\sum_{i=1}^n s_i$, and let $S_{s\concat m}=S(y_{s\concat m},1/4^{L+m})$. Again note that the top--right tile of $S_{s\concat m}$ has its top edge coinciding with the bottom edge of $c_{s\concat m}$.
\end{itemize}
Assume, for this paragraph only, that $\tau=\tau_c$ is the complete tree, and $T$ is the solid tile. For any $s$ in $\tau$, let $\tau_{s} = \{ s' \in \tau : s'$ extends $s\}$, and let   $B_{s}=\bigcup \{ S_s : s \in \tau_{s}\}$. By construction, $B_{(1)}$ is a $1/4$th copy of $B(T,\tau)=B_{()}$, and $B_{(2)}$ is a $1/16$th copy. It is easy to check that the height (in the $y$-coordinate) of $B_{()}$ is exactly $10/3$. So the height of $B_{(2)}$ is $1/16$th of this, which is $5/24$. The gap between the top edge of $B_{(1)}$ and the top edge of $B_{(2)}$ is $9/24$. Thus $B_{(2)}$ is disjoint from $B_{(1)}$. By self--similarity it follows that $B_{s}$ and $B_{t}$ meet if and only if one of $s$ and $t$ is an immediate successor of the other. This all shows that, for any tree and any tile, $B(T,\tau)$ is well defined, and is the edge connected union of tiles meeting along matching edges.

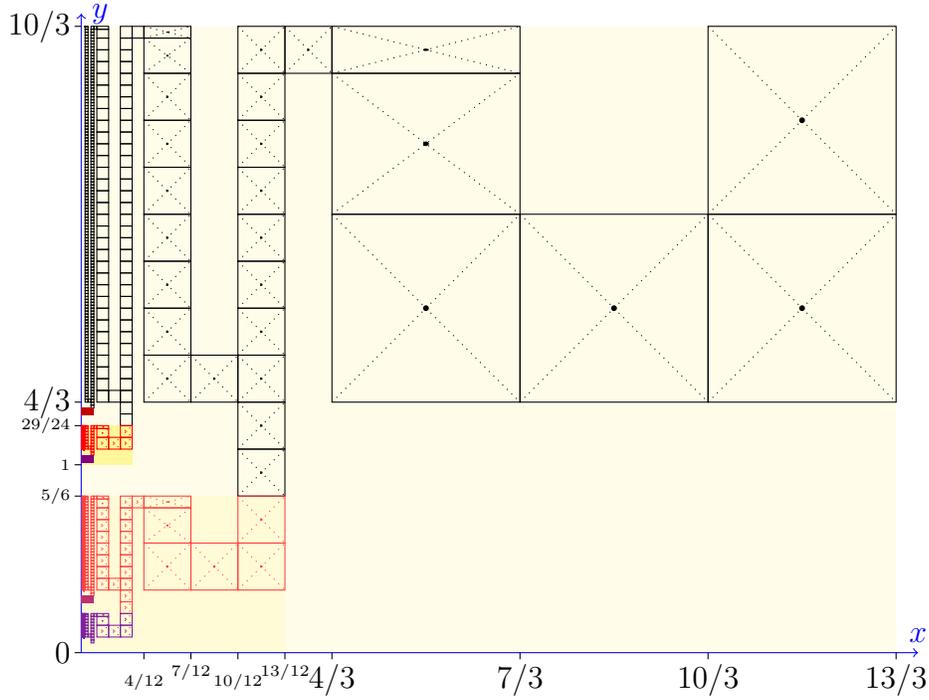
\tikzset{ptop/.style={}}
\tikzset{p1/.style={}}
\tikzset{p2/.style={}}
\tikzset{p3/.style={}}
\tikzset{p11/.style={}}
\tikzset{p12/.style={}}
\tikzset{p13/.style={}}
\tikzset{p21/.style={}}
\newcommand{\BSpace}[1]{
\begin{scope}[ptop]
\gSineCurve{5/6+5/2}{1}{#1}
\end{scope}
\begin{scope}[p1] 
\gSineCurve{5/6+5/2-2-2/4}{1/4}{#1}
\end{scope}
\begin{scope}[p2]
\gSineCurve{5/6+5/2-2-2/16}{1/16}{0}
\end{scope}
\begin{scope}[p3]
\fill (0,5/6+5/2-2-2/64) rectangle +(13/192,-5/128);
\end{scope}
\begin{scope}[p11] 
\gSineCurve{5/6+5/2-2-2/4-2/4-2/16}{1/16}{0}
\end{scope}
\begin{scope}[p12]
\fill (0,5/6+5/2-2-2/4-2/4-2/64) rectangle +(13/192,-5/128);
\end{scope}
\begin{scope}[p21]
\fill (0,5/6+5/2-2-2/16-2/16-2/64) rectangle +(13/192,-5/128);
\end{scope}
}

\begin{center}
\beginpgfgraphicnamed{fig2_Branch}
\begin{tikzpicture}[xscale=2.5, yscale=2.5]

\fill[fill=yellow!10](0,0) rectangle (13/3,5/2+5/6);
\fill[fill=yellow!20] (0,0) rectangle (13/12,5/6);
\fill[fill=yellow!50] (0,5/6+1/2-1/8-5/24) rectangle (13/48,5/6+1/2-1/8);

\draw[->] [blue] (0,0) -- (4.45,0) node[above] {$x$};
\draw[->] [blue] (0,0) -- (0,3.4) node[right] {$y$};

\foreach \pos in {4/3,7/3,10/3,13/3}
	\draw[shift={(\pos,-1pt)}] (0pt,0pt) -- (0pt, 1pt) node[below] {$\pos$};
\foreach \pos in {7/12,13/12}
	\draw[shift={(\pos,-1pt)}] (0pt,0pt) -- (0pt, 1pt) node[below] {\tiny{$\pos$}};
\foreach \pos in {4/12,10/12}
	\draw[shift={(\pos,-1pt)}] (0pt,0pt) -- (0pt, 1pt) node[below=4pt] {\tiny{$\pos$}};
\foreach \pos in {0,4/3,10/3}
         \draw[shift={(-1pt,\pos)}] (0pt,0pt) -- (1pt, 0pt) node[left] {$\pos$};
\foreach \pos in {5/6, 1,29/24}
         \draw[shift={(-1pt,\pos)}] (0pt,0pt) -- (1pt, 0pt) node[left] {\tiny{$\pos$}};

\begin{scope}[p1/.style={red!75!white}, p2/.style={red}, p3/.style={red!75!black}, p11/.style={red!75!white!50!blue},p12/.style={red!75!white!75!blue}, p21/.style={red!50!blue}]
\BSpace{0}
\end{scope}
\end{tikzpicture}
\endpgfgraphicnamed
\end{center}

We call the point $(0,10/3,0)$ the top left corner of $B(T,\tau)$. For $y$ in $\R$ and $a>0$, let $B(y,a,T,\tau)$ be $B(T,\tau)$ scaled in the $y$-coordinate only by $a$, and translated so its top left corner is at $(0,y,0)$.
\newcommand{\gBSpace}[3]{
\begin{scope}[cm={1,0,0,#2,(0,#1-#2-#2-#2-#2/3)}]
\BSpace{#3}
\end{scope}
}

Now our $K_\tau$ will consist of $\bigcup_{n \ge 0} B_n \cup L \cup S$, where
$B_n=B(y_n,1/2^n, T_{n+1}, \tau)$, for $y_n=7/2^n$, and the  two pieces $L$ and $S$ are defined as follows.

The set $L$ is a homeomorphic copy of the tile $T_0$, bent in the middle so that its base is contained in the $L$-shaped area $\{(x,y,0)\in\R^3 : -2/3\leq x\leq 0, -1\leq y\leq 7$  or $0\leq x\leq \frac{22}{3}, -1\leq y\leq 0\}$
and  the boundary of the base of the tile is the boundary of this area.


The set $S$ is a sine curve variant based on the tile $T_1$, which connects the branch spaces $B_n$, and converges down to the $x$-axis. Concretely, $S=\bigcup_{n \ge 0} (AR_n \cup D_n \cup AL_n \cup C_n)$ where
\begin{eqnarray*}
AR_n&=& \bigcup_{i=0}^{3 \cdot 4^n-1} T_1 ((13/3+i/4^{n},7/2^n), 1/4^n), \\
D_n &=& \bigcup_{i=1}^{3 \cdot 2^n -1} T_1((13/3+3-1/4^n,7/2^n-i/4^n),1/4^n), \\
AL_n&=& \bigcup_{i=1}^{2 \cdot 4^n-2}   T_1((13/3+1+i/4^n,7/2^n +1/4^n-3/2^n),1/4^n,1/4^n) \\
    & & \cup \, T_1((13/3+1,7/2^n +1/4^n-3/2^n),1/4^{n+1},1/4^n) \\
    & & \cup \, T_1((13/3+1+1/4^n,7/2^n+1/4^n-3/2^n),3/4^{n+1},1/4^n),  \quad \text{and} \\
C_n &=& \bigcup_{i=1}^{2^{n+1}} T_1(13/3+1,7/2^{n+1}+i/4^{n+1}, 1/4^{n+1}). 
\end{eqnarray*}

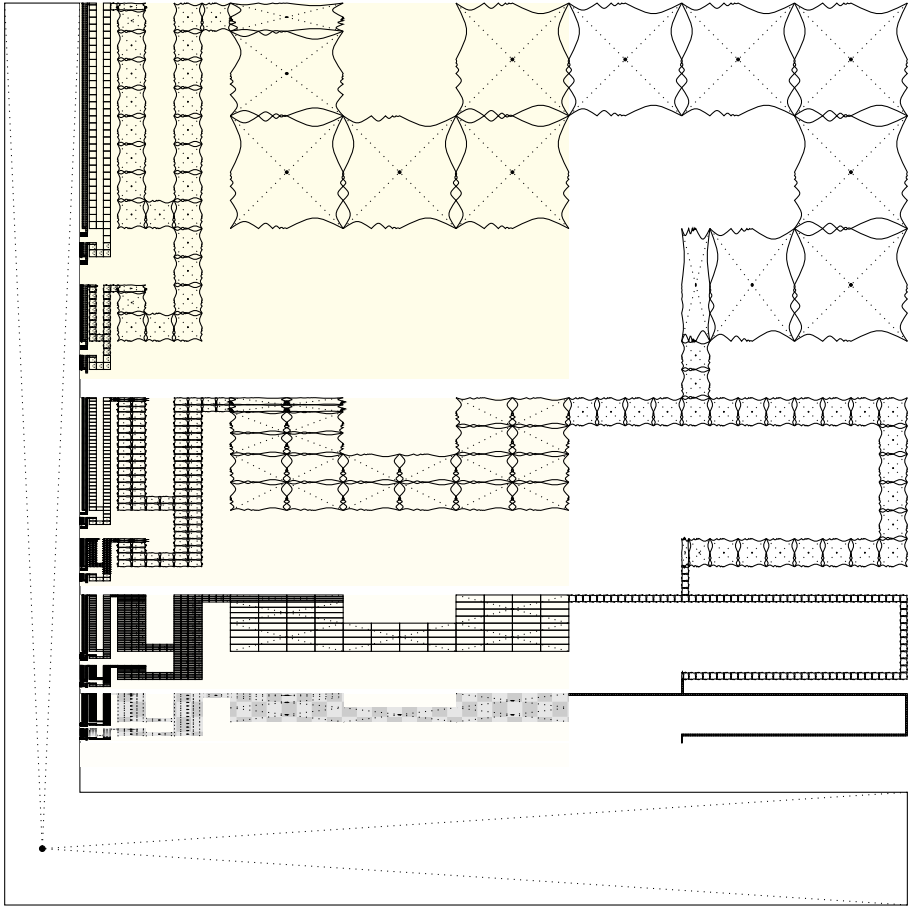
\beginpgfgraphicnamed{fig3_Ktau}
\noindent \begin{tikzpicture}[scale=1.5]
\draw (0,0) -- (0,7) -- (-2/3,7) -- (-2/3,-1) -- (22/3,-1) -- (22/3,0) -- cycle;
\draw[dotted] (-2/3,7) -- (-1/3,-1/2);
\draw[dotted] (0,7) -- (-1/3,-1/2);
\draw[dotted] (22/3,0) -- (-1/3,-1/2);
\draw[dotted] (22/3,-1) -- (-1/3,-1/2);
\draw[fill]  (-1/3,-1/2) circle (0.025);
\fill [fill=yellow!10] (0,7) rectangle ++(13/3,-10/3);

\gBSpace{7}{1}{1}

\fill [fill=yellow!6.66] (0,7/2) rectangle ++(13/3,-5/3);

\gBSpace{7/2}{1/2}{2}

\fill [fill=yellow!4.4444] (0,7/4) rectangle ++(13/3,-5/6);
\gBSpace{7/4}{1/4}{3}

\fill [fill=yellow!2.96296] (0,7/8) rectangle ++(13/3,-5/12);
\gBSpace{7/8}{1/8}{4}

\fill [fill=yellow!1.975] (0,7/16) rectangle ++(13/3,-5/24);

\gtileP{13/3}{7}{1}{1}{1}
\gtileP{13/3+1}{7}{1}{1}{1}
\gtileP{13/3+2}{7}{1}{1}{1}
\foreach \i in {1,...,2}
  {\gtileP{13/3+3-1}{7-\i}{1}{1}{1}}
\gtileP{13/3+1}{7-3+1}{1/4}{1}{1}
\gtileP{13/3+1+1/4}{7-3+1}{3/4}{1}{1}
\foreach \i in {1,...,2}
  {\gtileP{13/3+1}{7/2+\i/4}{1/4}{1/4}{1}}

\foreach \i  in {0,...,11}
  {\gtileP{13/3+\i/4}{7/2}{1/4}{1/4}{1}}
\foreach \i in {1,...,5}
  {\gtileP{13/3+3-1/4}{7/2-\i/4}{1/4}{1/4}{1}}
\gtileP{13/3+1}{7/2+1/4-3/2}{1/16}{1/4}{1}
\gtileP{13/3+1+1/16}{7/2+1/4-3/2}{3/16}{1/4}{1}
\foreach \i in {1,...,6}
  {\gtileP{13/3+1+\i/4}{7/2+1/4-3/2}{1/4}{1/4}{1}}
\foreach \i in {1,...,4}
  {\gtile{13/3+1}{7/4+\i/16}{1/16}{1/16}{1}}

\foreach \i  in {0,...,47}
  {\gtile{13/3+\i/16}{7/4}{1/16}{1/16}{1}}
\foreach \i in {1,...,11}
  {\gtile{13/3+3-1/16}{7/4-\i/16}{1/16}{1/16}{1}}
\gtile{13/3+1}{7/4+1/16-3/4}{1/64}{1/16}{1}
\gtile{13/3+1+1/64}{7/4+1/16-3/4}{3/64}{1/16}{1}
\foreach \i in {1,...,30}
  {\gtile{13/3+1+\i/16}{7/4+1/16-3/4}{1/16}{1/16}{1}}
\foreach \i in {1,...,8}
  {\gtile{13/3+1}{7/8+\i/64}{1/64}{1/64}{0}}

\foreach \i  in {0,...,191}
  {\gtile{13/3+\i/64}{7/8}{1/64}{1/64}{0}}
\foreach \i in {1,...,23}
  {\gtile{13/3+3-1/64}{7/8-\i/64}{1/64}{1/64}{0}}
\foreach \i in {1,...,16}
  {\gtile{13/3+1}{7/16+\i/256}{1/256}{1/256}{0}}

\gtile{13/3+1}{7/8+1/64-3/8}{1/256}{1/64}{0}
\gtile{13/3+1+1/256}{7/8+1/64-3/8}{3/256}{1/64}{0}
\foreach \i in {1,...,126}
  {\gtile{13/3+1+\i/64}{7/8+1/64-3/8}{1/64}{1/64}{0}}

\end{tikzpicture}
\endpgfgraphicnamed

\begin{claim}
$K_\tau$ is a rational continuum.
\end{claim}

\noindent\textit{Proof: } Let $R=\bigcup_n B_n \cup S$. Let $L' =\{(x,y,0)\in\R^3 : x=0, -1\leq y\leq 7\text{ or }-2/3\leq x\leq 22/3, y= 0\}$, be the inner boundary of the base of $L$. 

Since $\cl{R} \subseteq R \cup L'$, $K_\tau$ is clearly compact. Since $L$ and $R$ are connected, and $S$ is a variant topologists sine curve, clearly $K_\tau$ is connected.

For all the points of $K_\tau$ except those on $L'$,
we have a natural neighborhood base at the point for which each element has a countable boundary (which comes from the tile(s) the point is in).

Take any point $\vx$ in $L'$. We suppose now, $\vx=(x_0,0,0)$ (the other case is similar).
Because $B_n$ is based on the tile $T_n$, combined with the fact that the $T_1$s in the connecting sine curve, $S$, have size shrinking to zero, the set $M$ of all $x$-components of the left and right edges of the base of tiles in $R$ is dense in $[0,22/3]$.

Let $U$ be a rectangular neighborhood  of $\vx$ in $\R^3$, and $r_{min}=\min\{x : (x,0,0)\in U\}$ and $r_{max}=\max\{x : (x,0,0)\in U\}$. 
Without loss of generality, if $y_{max}$ is the value of the maximum $y$-component in $U$ then $\set{(x,y,z)\in U}{y=y_{max}}$ do not intersect with any of the 
$B_n$, i.e. the top of $U$ is in between $B_n$ and $B_{n+1}$ for some $n$. 

The set $U\cap L$ includes a neighborhood $N$ of $x$ which has countable boundary. Let $a=\min\{x : (x,0,0)\in N\}$ and $b=\max\{x : (x,0,0)\in N\}$. 
Since $M$ is dense there are sequences $(a_n)_n$ and $(b_n)_n$ in $M$ such that $a_n$ increases to $a$, $b_n$ decreases to $b$, and for each $n$, $r_{min}\leq a_n\leq a<b\leq b_n\leq r_{max}$. 
Let $m_1\geq n$ be such that both of the lines $x=a_1$ and $x=b_1$ intersect the $xy$-projection of $B_{m_1}\cup \{(x,y,z)\in S: y\leq 7/2^{m_1}\}$ along edges of tiles only. Let $S_{n}=\{(x,y,z)\in S: y\leq 7/2^n\}$.
And inductively, let $m_i\geq m_{i-1}$ such that the lines $x=a_i$ and $x=b_i$ intersect the $xy$-projection of $B_{m_i} \cup S_{m_i}$ along edges of tiles only.
Now take 
$N'=\bigcup_i((S_{m_i}\setminus S_{m_{i+1}}\cup\bigcup_{k+m_i<m_{i+1}} B_{m_i+k})\cap\{(x,y,z): a_i\leq x\leq b_i\})$. 

\noindent Here for each $i$, we cut $B_{m_i+k}$ along edges of finitely many tiles, hence the boundary is at most countable. And similarly for $S_{m_i}\setminus S_{m_{i+1}}$, we cut along the edges of finitely many tiles. Thus $N'$ has countable boundary. Moreover, $N\cup N'\subset R$ is a neighborhood of $\vx$ with countable boundary.
\subqed

\medskip

\begin{claim}
If $\tau$ has an infinite branch then $K_\tau$ is $\omega$-sac.
\end{claim}

\noindent\textit{Proof: } Suppose $\tau$ has an infinite branch.
Note that if $T$ is any tile,  then there is a branch of edge connected tiles in $B(T, \tau)$ which converges to a point $\vy_{\tau}$ on the $y$-axis. 

We first show that for any $m \ge 1$, the branch space $B(T_m,\tau)\cup\{\vy_{\tau}\}$ is $\omega$-sac. To do so we only need to check that if $\vy_\tau$ is one of the $n$-points $x_1,\ldots,x_n$ in $B(T_m,\tau)\cup\{\vy_{\tau}\}$, then we can find an arc joining them in that order. 
Suppose $x_k=\vy_\tau$. Then the points $x_1,\ldots,x_{k-1},x_{k+1},\ldots,x_n$ are in some finite family  of edge connected tiles of $B(T_m,\tau)$. Let $t$ be a tile in the branch space $B(T_m,\tau)$ such that none of the $x_i$s is in the tiles to the left and bottom of this tile except for $x_k=\vy_\tau$. Let $\vy_1$ be the bottom left corner of $t$, and $\vy_2$ be the top left corner of $t$. 
Then by Lemma~\ref{edge_conn_T1} there is an arc $\alpha_0$ in $B(T_m,\tau)$ through the points $x_1,\ldots,x_{k-1},\vy_1,\vy_2,x_{k+1},\ldots,x_n$ in the given order. Let $\alpha_1$ be the part of $\alpha_0$ through $x_1,\ldots, x_{k-1},\vy_1$. Let $\beta_1$ be the arc starting at $\vy_1$ and ending at $\vy_\tau$ obtained by traveling along the right and bottom edges of tiles of the branch converging to $\vy_\tau$. Similarly, let $\alpha_2$  be the part of $\alpha_0$ through $y_2,x_{k+1},\ldots, x_{n}$. And let $\beta_2$ be the arc starting at $\vy_\tau$, and following  the left and top edges of tiles of the branch converging to $\vy_\tau$, back to $\vy_2$. Then the arc $\alpha$ obtained by following $\alpha_1$, $\beta_1$, $\beta_2$ and then $\alpha_2$, is the desired arc through the points $x_1,\ldots,x_n$ in the given order.

Back now to $K_\tau$, when $\tau$ has an infinite branch. For each $n$, $B_n$ a has a corresponding branch converging to a point $\vy_{n}$ on the $y$-axis. 
An easy modification of the argument for $B(T_m,\tau)\cup\{\vy_{\tau}\}$ shows that the space $R \cup \{\vy_{n} : n \in \N\}$ is $\omega$-sac.

Since $R \cup \{\vy_n\}_n$ is $\omega$-sac, $(R \cup \{\vy_n\}_n) \cap L = \{\vy_n\}_n$, and  $L$ is  $\omega$-sac$^+$ with respect to discrete sets (Lemma~\ref{tiles_om_sac_plus}, part (i)), it follows from the $\omega$--Gluing Lemma that $K_\tau=R \cup L$ is indeed $\omega$-sac. \subqed

\medskip

\begin{claim}
If $\tau$ has no infinite branch then $K_\tau$ is not $2$-sac.
\end{claim}

\noindent\textit{Proof: }
If $\tau$ does not have any infinite branches, then there are no arcs connecting $L$ to $R$. This is clear because, without infinite branches, any path starting in $R$ and attempting to reach $L$ is forced to travel along a topologist's sine curve variant --- which is impossible.
\subqed

\medskip

\begin{claim}
The map $\tau \mapsto K_\tau$ is continuous.
\end{claim}
\noindent\textit{Proof: } Let $K : \mathbf{Tr} \to \mathcal{C}(\R^3)$ given by $K(\tau)=K_\tau$. Let $s$ be in $\N^{<\N}$, and write $[s]$ for the set of all trees containing $s$. Then $[s]$ is a closed and open subset of $\mathbf{Tr}$. Subbasic open sets in $\mathcal{C}(\R^3)$ are of one of two forms: (i) $\langle U \rangle = \{ C : C \subseteq U\}$ and (ii) $\langle X;V \rangle = \{ C : C \cap V \ne \emptyset\}$, where $U$ and $V$ are open subsets of $\R^3$. We show inverse images under $K$ of both types of subbasic open set are open in $\mathbf{Tr}$, thus confirming continuity of the map $\tau \mapsto K_\tau$.

For subbasic sets of type (ii), the sets $V$ may be taken to come from any basis for $\R^3$; we will take for $V$ open balls in $\R^3$ which either meet, or have closure disjoint from, $L \cup S$. Fix such a $V$. If $V$ meets $L \cup S$, then $K^{-1} \langle X;V\rangle = \mathbf{Tr}$. If the closure of $V$ is disjoint from $L$, then for any tree $\tau$, $V$ meets only finitely many $B_n(\tau)$, and in each of these branch spaces, meets only finitely many sine curves. Suppose $V$ meets sine curves labeled by $s_1, \ldots , s_k$. Then $K^{-1} \langle X;V\rangle = \bigcup \{ [s_i] : 1 \le i \le k\}$, which is open (each $[s_i]$ is open).

For subbasic sets of type (i), if $L \cup S$ is not contained in $U$, then $K^{-1} \langle U\rangle = \emptyset$. So suppose, $L \cup S \subseteq U$. Let $\tau_c$ be the complete tree. Then all but finitely many of the sine curves making up the $B_n(\tau_c)$s are contained in $U$. Let them be labeled by $s_1, \ldots , s_k$. Then $K^{-1} \langle U\rangle = \mathcal{C}(\R^3) \setminus \bigcup \{ [s_i] : 1 \le i \le k\}$, which is open (each $[s_i]$ is closed).
\subqed

\medskip

Claims 1--4 show that $\tau \mapsto K_\tau$ is as required --- continuous, and $K_\tau$ is a rational continuum which is $\omega$-sac if $\tau$ has an infinite branch, but is not even $2$-sac when $\tau$ has no infinite branches.

\bigskip

We now turn to the case for $R_{n, \neg (n+1)}$. To start fix $n \ge 2$. Select $n-2$ points $a_1, \ldots, a_{n-2}$ from the interior of the right hand edge of the base of $T_0$. Similarly to the definition of $T_1$, shrink $T_0$ while keeping fixed the set  $\{a_1, \ldots, a_{n-2}\}$ and the top edge of the base. This gives a tile $\widehat{T}_n$. Now consider the map $\tau \mapsto K^n_{\tau}$ where $K^n_{\tau}$ is $K_\tau$ along with the tile $\widehat{T}_n(22/3,0,1)$. Then it is easy to see (given our previous work) that  $K^n_{\tau}$ is a rational continuum and the map $\tau \mapsto K^n_{\tau}$ is continuous. Because the extra tile, $\widehat{T}_n(22/3,0,1)$, meets the rest of $K^n_{\tau}$ in exactly $n-1$ points (namely $a_1, \ldots, a_{n-2}$ and the topleft corner of the base of the tile), $K^n_{\tau}$ is never $(n+1)$-sac. When $\tau$ has no infinite branch, then $K^n_{\tau}$ is not $2$-sac, so definitely not in $R_{n, \neg (n+1)}$. But when $\tau$ has an infinite branch, both  $\widehat{T}_n(22/3,0,1)$ and the rest of $K^n_{\tau}$ are $\omega$-sac, and (again) meet in $n-1$ points --- so by Lemma~\ref{finite_gluing} is $n$-sac. 
\end{proof}

By definition of $\omega$-sac, the set of all $\omega$-sac curves is a $\piset{2}{1}$ set. It turns out there is no simpler characterization for being $\omega$-sac. Given the previous theorem, the proof of the following result is very similar to that of Becker and (independently) Ajtai that the set of arc connected subcontinua of $\R^3$ is $\mathbf{\Pi}_2^1$-complete  (see \cite{kech} 37.11). Consequently we give just a sketch of the proof, highlighting the differences.

\begin{theorem}
The sets $S_n$ of $n$-sac curves, for all $n\geq 2$ and $n=\omega$, are  $\piset{2}{1}$-complete subsets of the space $\mathcal{C}(\R^N)$, where $N\geq 4$.
\end{theorem}

\begin{proof}
We prove the claim, for all $n$ and $N$ simultaneously, by proving that there is a continuous map $\Phi$ from the space $\N^\N$ into the space $\mathcal{C}(\R^4)$ such that: given a $\mathbf{\Pi}_2^1$ set $A\subset\N^\N$, if $x\in A$ then $\Phi(x)=P_x$ is a curve which is not arc connected (i.e. $2$-sac), while if $x\in A$, then $P_x$ is an $\omega$-sac curve.

Let $A$ be a $\piset{2}{1}$ subset of $\N^\N$ and $B$ be a $\mathbf{\Sigma}_1^1$ subset of $\N^\N\times 2^\N$ with $x\in A$ if and only if for each $y\in 2^\N$, $(x,y)\in B$.
Now let $\tau$ be a tree on $\N\times 2\times \N$ with 
$B=\{(x,y) : \exists z\in\N^\N (x,y,z)\text{ is a branch of } \tau\}=\{(x,y) : \tau(x,y)\in\mathbf{IF}\}$. Recall that $\tau(x,y)=\{s:(x\restrict \mathop{length}(s),y\restrict \mathop{length}(s),s)\in\tau\}$ is a tree on $\N$.

Now for each $x\in\N^\N$, we will construct a curve $P_x\subset\R^4$ as follows.
 First, we identify the Cantor space $2^\N\subset\N^\N$ with the standard Cantor set in $[0,1]$. 
Then for each $y\in 2^\N$, let $L_{x,y}=K_{\tau(x,y)}$, as described in Theorem~\ref{rationalsac}, placed in the cube $\{(a,b,c,d) : -\frac{2}{3}\leq a\leq \frac{22}{3}, -1\leq b\leq 7, c\geq 0, d=y\}$. Thus the outside edges of the tile $L$ in $K_{\tau(x,y)}$ is on $a=-\frac{2}{3}$ or $b=-1$. 
Let $P_x=\bigcup_{y\in 2^\N}L_{x,y}\cup M$.
Here we have  connected the continua $L_{x,y}$ along the edges on $a=-\frac{2}{3}$ by adding the curve $M$, which is a copy of the Menger cube scaled and translated inside the cube $\{(a,b,c,d) : -2\leq a\leq -\frac{2}{3},-1\leq b\leq 7,c=0,0\leq d\leq 1\}$.
 
Then, $P_x$ is a curve  and the map $x\mapsto P_x$ from $\N^\N\rightarrow \mathcal{C}(\R^4)$ is continuous. Moreover, $x\in A$ if and only if for each $y\in 2^\N$, $\tau(x,y)\in\mathbf{IF}$. Thus if $x\notin A$, then there is $y\in 2^\N$ with $\tau(x,y)\notin \mathbf{IF}$, so the corresponding rational $\omega$-sac continua $L_{x,y}=K_{\tau(x,y)}$ is not 2-sac, hence the union $P_x=M\cup\bigcup_{y\in 2^\N} L_{x,y}$ is also not 2-sac. 

On the other hand, if $x\in A$, then for each $y\in 2^\N$, $L_{x,y}$ is $\omega$-sac by Theorem~\ref{rationalsac}. It is straightforward to check that the Menger cube $M$ is $\omega$-sac. The proof that $P_x$ is $\omega$-sac is now straightforward given the techniques developed for Theorem~\ref{rationalsac}.
\end{proof}

Since the set of $\omega$-sac curves is $\mathbf{\Pi}_2^1$-complete it is not a $\mathbf{\Sigma}_1^1$-set (i.e. an analytic set). From which the next corollary easily follows.
\begin{cor}
There is no mapping universal for the class of all $\omega$-sac curves. More generally, given any countable family of $\omega$-sac curves, $K_n$ for $n$ in $\N$, there is an $\omega$-sac curve $L$ which is not the continuous image of any $K_n$.
\end{cor}

\section*{Open Problems} Our results raise a number of questions about optimality. For a space $X$, define the \emph{sac number} of $X$, denoted $\mathop{sac}(X)$, to be  the maximal $n$ such that $X$ is $n$-sac, or $\infty$ if $X$ is $\omega$-sac. 
\begin{itemize}
\item What is the sac number of the $N$-trix? We know $\mathop{sac}(3\text{-trix})=\mathop{sac}(4\text{-trix})=3$ but $\mathop{sac}(5\text{-trix})=5$, and in general $\mathop{sac}(N\text{-trix}) \le N$.
\item What is the sac number of a closed disk with $N$-handles? What is the sac number of (compact) $2$-manifolds? What is the sac number of finite simplicial $2$-complexes?
\end{itemize} 

We have shown that the $3$-sac graphs can be simply characterized, but (for $n \ge 2$ or $n=\omega$) there is no simple characterization of $n$-sac rational continua.
\begin{itemize}
\item Characterize (simply) the regular $n$-sac curves, or prove that no such characterization is possible.
\end{itemize}

The $n$-sac property is a very natural strengthening of arc connectedness. Whenever a space is known to be arc connected (i.e. $2$-sac) we are lead to ask for which $n$ it is $n$-sac. For example the hyperspaces $2^X$ and $\mathcal{C}(X)$ of compact subsets and subcontinua of a continuum $X$ are well known to be arc connected.
\begin{itemize}
\item When is $2^X$ or $\mathcal{C}(X)$ $n$-sac or $\omega$-sac? We know that for hereditarily indecomposable $X$ the space of subcontinua is not $3$-sac. 
\item What about when we restrict to, say, locally connected continua?
\item Are there, for each $n$, continua which are $n$-sac but not $(n+1)$-sac? 
\end{itemize}

\end{document}